\documentclass[10pt]{amsart} 

\usepackage{verbatim}

\usepackage{amscd,amsmath,amssymb}
\usepackage[mathscr]{eucal}
\usepackage{color}
\usepackage{graphicx}

\theoremstyle{plain}

\numberwithin{equation}{section}

\newtheorem{theorem}{Theorem}[section]
\newtheorem{lemma}[theorem]{Lemma}
\newtheorem{proposition}[theorem]{Proposition}
\newtheorem{corollary}[theorem]{Corollary}

\newtheorem{definition}[theorem]{Definition}

\theoremstyle{remark}
\newtheorem{remark}[theorem]{Remark}
\newtheorem*{claim*}{Claim}
\newtheorem*{example*}{Example}
\newtheorem{example}[theorem]{Example}
\newtheorem*{remark*}{Remark}

\newcommand{\R}{\mathbf{R}}
\newcommand{\Rn}{\R^n}

\newcommand{\U}{\mathcal{U}}
\newcommand{\V}{\mathcal{V}}

\newcommand{\dom}{\mathop{\rm Dom}}

\newcommand{\intr}{\mathop{\rm int}}

\newcommand\bexp{\hbox{$b$-Exp}}

\newcommand\Bzero{{\bf (B0)}}
\newcommand\Bone{{\bf (B1)}}
\newcommand\Btwo{{\bf (B2)}}
\newcommand\Bthree{{\bf (B3)}}

\newcommand\AthreeW{{(A3w)}}
\newcommand\BthreeU{{\bf (B3u)}}
\newcommand\BthreeS{{\bf (B3s)}}
\newcommand\cl{\overline}
\newcommand\X{X}
\newcommand\Y{Y}
\newcommand\tb{{\tilde b}}
\newcommand\tbs{{\tilde b^*}}
\newcommand\bs{{b^*}}
\newcommand\tc{{\tilde c}}

\newcommand\ur{{u_\emptyset}}
\newcommand\tu{{\tilde u}}
\newcommand\N{{\bf N}}
\newcommand\spt{{\rm spt}\thinspace}
\newcommand\e{{\epsilon}}
\renewcommand\i{{\infty}}
\newcommand\tf{{\tilde f}}

\newcommand\tL{{\tilde L}}
\newcommand\ta{{\tilde a}}

\newcommand\tY{{\tilde Y}}
\newcommand\zero{{\bf 0}}
\newcommand\tU{{\tilde U}}
\newcommand\Prob{{\mathscr P}}
\newcommand\dummy{}

\begin{document}

\title[When is multidimensional screening a convex program?]
{When is multidimensional screening a convex program?$^*$
}

\author{Alessio Figalli$^\dagger$, Young-Heon Kim$^\ddag$ and Robert J. McCann$^\S$
}




\date{\today}

\thanks{
{\em Journal of Economic Literature Classification Codes:} C61, D82, D42, H21, J42. 
\\ $^*$The authors are 
grateful to
the Institute for Pure and Applied Mathematics at UCLA and
the Institut Fourier at Grenoble
for their generous hospitality during various stages of this work.
[RJM]'s research was supported in part by NSERC grant 217006-08 and
NSF grant DMS-0354729. [YHK] is supported partly by
NSF grant DMS-0635607 through the membership at
Institute for Advanced Study, Princeton NJ and NSERC
discovery grant 371642-09. Any opinions, findings and conclusions or
recommendations expressed in this material are those of authors and do not reflect the
views of the Natural Sciences and Engineering Research Council of Canada (NSERC) or
of the US National Science Foundation (NSF).
\hfill\copyright 2009 by the authors.
\\ \\ $^\dagger$Department of Mathematics, University of Texas, Austin TX 78712 USA
{\tt figalli@math.utexas.edu}
\\ $^\ddag$Department of Mathematics, University of British Columbia, Vancouver BC V6T 1Z2 Canada
{\tt yhkim@math.ubc.ca}
\\ $^\S$Corresponding author: Department of Mathematics, University of Toronto, Toronto, Ontario Canada M5S 2E4
{\tt mccann@math.toronto.edu}.
}


\begin{abstract}
A principal wishes to transact business with a multidimensional distribution of agents
whose preferences are known only in the aggregate.  Assuming 
a twist (= generalized Spence-Mirrlees single-crossing) hypothesis
{\dummy and that agents can choose only pure strategies},
we identify a structural condition on the preference
$b(x,y)$ of agent type $x$ for product type $y$ --- and on the principal's costs $c(y)$ --- which is
necessary and sufficient for reducing the profit maximization problem faced by the principal
to a convex program.  This is a key step toward making the principal's problem
theoretically and computationally tractable;  in particular, it allows us to derive uniqueness
and stability of the principal's optimum strategy {--- and similarly of the strategy
maximizing the expected welfare of the agents when the principal's profitability is constrained.}
We call this condition non-negative cross-curvature: it is also
(i) necessary and sufficient to guarantee convexity of the set of $b$-convex functions,
(ii) invariant under reparametrization of agent and/or product types by diffeomorphisms,  and
(iii) a strengthening of Ma, Trudinger and Wang's necessary and sufficient condition
{(A3w)}
for continuity of the correspondence between an exogenously prescribed distribution
of agents and of products.
We derive the persistence of economic effects such as the desirability for a monopoly
to establish prices so high they effectively exclude a positive fraction of its potential
customers, in nearly the full range of non-negatively cross-curved models.
\end{abstract}

\maketitle
\tableofcontents

\section{Introduction}
\label{S:introduction}

The principal-agent paradigm provides a microeconomic framework
for modeling non-competitive decision problems which must be made in the face of
informational asymmetry.  Such problems range from monopolist nonlinear pricing
\cite{MussaRosen78} \cite{Spence80}
\cite{GoldmanLelandSibley84}
\cite{Wilson93} \cite{Armstrong96}
and product line design (``customer screening'')
{\cite{RothschildStiglitz76} \cite{MaskinRiley84} \cite{RochetChone98},
to optimal taxation \cite{Mirrlees71} \cite{Mirrlees76},
labour market signalling and contract theory \cite{Spence73} \cite{Spence74} \cite{AraujoGottliebMoreira07},
regulation of  monopolies \cite{BaronMyerson82} \cite{McAfeeMcMillan87} \cite{LewisSappington88} \cite{LaffontTirole93}
\cite{Armstrong99} including public utilities  \cite{Roberts79} \cite{BrownSibley86},
and mechanism design \cite{GreenLaffont77} \cite{McAfeeMcMillan88} \cite{MonteiroPage98}.}
A typical example would be the problem faced by a monopolist who wants to market
automobiles $y \in \Y$ to a population of potential buyers (``agents'') $x \in \X$.
Knowing the preferences $b(x,y)$ of buyer $x$ for car $y$,  the relative frequency $d\mu(x)$
of different buyer types in the population,  and the cost $c(y)$ she
incurs in manufacturing car type $y$,  the principal needs to decide which
products (or product bundles) to manufacture and how much to charge for each of
them, so as to maximize her profits.

In the simplest models, {e.g. \cite{Spence73} \cite{RothschildStiglitz76},}
there are only a finite number of product possibilities
(e.g.\ with air conditioning, or without) and a finite number of buyer types
(e.g.\ rich, middle-class, and poor);  or possibly a one-dimensional continuum of product
possibilities (parameterized, say,  by quality) and of agent types
(parameterized, say, by income)
{\cite{Mirrlees71} \cite{Spence74} \cite{MussaRosen78} \cite{BaronMyerson82}.}
Of course,  real cars depend on more than
one parameter --- fuel efficiency, comfort, options, reliability, styling, handling and safety,
to name a few --- as do car shoppers, who vary in wealth, income, age, commuting needs,
family size, personal disposition, etc.  Thus realistic modeling requires multidimensional
type spaces $\X \subset \R^m$ and $\Y \subset \Rn$ as in \cite{McAfeeMcMillan88} \cite{Wilson91}
\cite{Mirrlees96} \cite{RochetStole03} \cite{Basov05}  \cite{DeneckereSeverinov09p}.
Although such models can often be reduced to optimization problems in the calculus of
variations \cite{Carlier01} \cite{Basov01}, in the absence of convexity they remain
dauntingly difficult to analyze.  Convexity --- whether manifest or hidden --- rules out critical
points other than global minima,  and is often the key to locating and characterizing {\dummy
optimal strategies} either numerically or theoretically.  The purpose of the present article is to
determine when convexity is present,  assuming the dimensions $m=n$ of the agent and product
type spaces coincide.


An archetypal model was addressed by Wilson \cite{Wilson93},
Armstrong \cite{Armstrong96}, and Rochet and Chon\'e \cite{RochetChone98}.
A particular example from the last of these studies
makes the simplifying hypotheses $\X = \Y = [0,\infty\mathclose[^n$, $c(y) = |y|^2/2$, and
$b(x,y) = \langle x, y \rangle$.  By assuming this {\em bilinearity} of buyer
preferences,   Rochet and Chon\'e were able to show that the principal's problem
can be reduced to a {\dummy quadratic} minimization over the set
of non-negative convex functions
--- itself a convex set.  Although the convexity constraint makes this variational
problem non-standard, for buyers distributed uniformly throughout
the unit square, they exploited a combination of
theoretical and computational analysis to show a number of results
of economic interest.  Their most striking conclusion was that the
profit motive alone leads the principal to discriminate between
three different types of buyers:  (i) low-end customers whom she
will not market cars to, because --- as Armstrong had already discovered ---
making cars affordable to this segment of the market would
cost her too much of her mid-range and high-end profits;  (ii) mid-range customers,  whom she will
encourage to choose from a one-parameter family of affordably-priced
compromise vehicles; (iii) high-end customers,  whom she will use
both available dimensions of her product space to market expensive
vehicles individually tailored to suit each customer's desires.
Whether or not such {\em bunching} phenomena are robust is an unanswered question
of considerable interest which --- due to their specificity to particular preference
functions --- the techniques of the foregoing authors remain unable to address.
The possibility of non-robustness was highlighted in \cite{Basov05};
below we go further to suggest which specific perturbations of the preference function $b(x,y)$
are most likely to yield robust results.  On the other hand, our conclusions confirm Armstrong's
assertion that what he called {\em the desirability
of exclusion} is a very general phenomenon in the models we study
(Theorem~\ref{T:Armstrong}).
This exclusion however, is not generic when the dimensions
of the type and allocation spaces differ \cite{DeneckereSeverinov09p}:
Deneckere and Severinov gave necessary and
sufficient conditions for exclusion when $(m,n)=(2,1)$.

For general preferences $b(x,y)$,  the principal's problem can be reformulated  as a minimimization problem over the space of $b$-convex functions
(Definition \ref{D:b-convex}), according to Carlier \cite{Carlier01}.
Such functions generally form a compact but non-convex set,  which prevented Carlier from deducing
much more than the existence of an optimal strategy for the principal
--- a result which can
also be obtained using the method of Monteiro and Page \cite{MonteiroPage98};
(for related developments see 
Basov \cite{Basov05} or Rochet and Stole \cite{RochetStole03}).
{\dummy Our present purpose
is to identify conditions on the agent preferences which guarantee convexity of this
feasible set (Theorem~\ref{T:convex set U}).  In the setting we choose,  the conditions we find
will actually be necessary as well as sufficient for convexity;  this necessity
imparts a significance to these conditions even if they appear unexpected
or unfamiliar.} If, in addition, the principal's manufacturing cost $c(y)$ is
$\bs$-convex, for $\bs(y,x) := b(x,y)$,
the principal's problem becomes a convex program which renders it much more
amenable to standard theoretical and computational techniques.  Although the resulting problem
retains the complexities of the Wilson, Armstrong, and Rochet and Chon\'e's models,
we are able to deduce new results
which remained inaccessible until now,  such as conditions guaranteeing uniqueness
{(Theorem~\ref{T:uniqueness}) and stability (Corollary~\ref{C:stability})
of the principal's optimum strategy.  The same considerations and results apply also to the
problem of maximimizing the total welfare of the agents under the constraint that it remain
possible for the principal to operate without sustaining a loss (Remark \ref{R:SWM under profitability constraint}).}

{The initial impetus for this study emerged from discussions with Ivar Ekeland.}
RJM is pleased to express his gratitude to Ekeland for introducing him to the principal-agent
problem in 1996,  and for anticipating already at that time that
it ought to be tackled using techniques from the mathematical theory of optimal transportation.
This approach was exploited by {Carlier \cite{Carlier01} in his doctoral thesis,}
following earlier works by Rochet \cite{Rochet85} \cite{Rochet87} and Rochet and
Chon\'e \cite{RochetChone98},  and was recently extended to a different
but related class of problems by Buttazzo and Carlier \cite{ButtazzoCarlier09p}.
We are grateful to Giuseppe Buttazzo and Guillaume Carlier also,
for stimulating discussions.


\section{Hypotheses: the basic framework}
\label{S:the model}

As in Ma, Trudinger and Wang's work concerning the smoothness of optimal mappings \cite{MaTrudingerWang05},
let us assume the buyer preferences satisfy the following hypotheses.
Let $\cl \X$ denote the closure of any given set $X\subset \R^n$, and
for each $(x_0,y_0) \in \cl \X \times \cl \Y$ assume:\\
\Bzero\ $b \in C^4\big(\cl \X \times \cl \Y\big)$, where
$\X \subset \Rn$ and $\Y \subset \Rn$ are open and bounded; \\
\Bone\ (bi-twist)
$\left.\begin{array}{c}
        y \in \cl \Y \longmapsto D_x b(x_0,y) 
\cr     x \in \cl \X \longmapsto D_y b(x,y_0) 
        \end{array}\right\}$ are 
        diffeomorphisms onto their ranges; \\
\Btwo\ (bi-convexity) 
$\left.\begin{array}{l}
        X_{y_0} := D_y b(\X,y_0) 
\cr     
        Y_{x_0} := D_x b(x_0, \Y) \cr
\end{array}\right\}$
are convex subsets of $\Rn$.\\
Here the subscript $x_0$ serves as a reminder that $Y_{x_0}$
denotes a subset of the cotangent space  $T^*_{x_0} \X {= \Rn}$
to $\X$ at $x_0$.
Note \Bone\ is strengthened form of the multidimensional
generalization \cite{Ruschendorf91} \cite{Gangbo95} \cite{Levin99}
of the Spence-Mirrlees single-crossing condition
expressed in R\"uschendorf, in Gangbo, and in Levin.
It asserts the marginal utility  of buyer type $x_0$ in equation \eqref{marginal utility}
determines the product he selects uniquely and smoothly,  and similarly that buyer type
who selects product $y_0$ will be a well-defined smooth function of $y_0$ and the
marginal cost of that product;
\Bone\ is much less restrictive than the generalized single crossing condition proposed by
McAfee and McMillan~\cite{McAfeeMcMillan88}, since the iso-price curves in the latter context
become hyperplanes, effectively reducing the problem to a single dimension. We also assume\\
\Bthree\ (non-negative cross-curvature)
\begin{equation}\label{MTW}
\frac{\partial^4}{\partial s^2 \partial t^2}\bigg|_{(s,t)=(0,0)} b(x(s),y(t)) \ge 0
\end{equation}
for each 
curve
$t \in[-1,1] \longmapsto (D_y b(x(t),y(0)),D_x b(x(0),y(t)))$ 
forming an affinely parameterized line segment in $\cl \X_{y(0)} \times \cl \Y_{x(0)} \subset \R^{2n}$.
If the inequality \eqref{MTW} becomes strict whenever 
$x'(0)$ and $y'(0)$ 
are non-vanishing, 
we say the preference function $b$ is {\em positively cross-curved},
and denote this by \BthreeU.

\begin{remark}[{\dummy Mathematical lineage}]
Condition \Bthree\ can alternately be defined as in Lemma \ref{L:t-convex DASM}
using Definition \ref{D:b-exp};  the convexity asserted by that lemma may appear more
intuitive and natural than \Bthree\ from point of view of applications.
Historically, non-negative cross-curvature arose
as a strengthening of Trudinger and Wang's
criterion \AthreeW\ guaranteeing smoothness of optimal maps in the Monge-Kantorovich
transportation problem \cite{TrudingerWang07p};  unlike us,
they require (\ref{MTW}) only if, in addition,
\begin{equation}\label{nullity}
\frac{\partial^2}{\partial s \partial t}\bigg|_{(s,t)=(0,0)} b(x(s),y(t)) = 0.
\end{equation}
Necessity of Trudinger and Wang's condition for continuity was shown by
Loeper \cite{Loeper07p}, who (like \cite{KimMcCann07p} \cite{Trudinger06})
also noted its covariance and some of its relations to {the geometric notion of}
curvature.
Their condition relaxes a hypothesis proposed with Ma \cite{MaTrudingerWang05},
which required strict positivity of (\ref{MTW}) when (\ref{nullity}) holds.
The strengthening considered here was first studied in a
different but equivalent form
by Kim and McCann,  where both the original and the modified
conditions were shown to
correspond to pseudo-Riemannian sectional curvature conditions induced by
buyer preferences on $\X \times \Y$,  thus highlighting their invariance under reparametrization
of either $\X$ or $\Y$ by diffeomorphism; see Lemma 4.5 of \cite{KimMcCann07p}.
Other variants and refinements
of Ma, Trudinger, and Wang's 
condition have been proposed and investigated
by Figalli and Rifford \cite{FigalliRifford08p} and
Loeper and Villani \cite{LoeperVillani08p} for different purposes at about the same time.

Kim and McCann showed non-negative cross-curvature
guarantees tensorizability of condition
{{\bf (B3)}},
which is useful for building examples of preference functions which satisfy it
\cite{KimMcCann08p}; in suitable coordinates, it guarantees convexity of each
$b$-convex function, as they showed with Figalli \cite{FigalliKimMcCann};
see Proposition \ref{P:convexity}.
Hereafter we show,  in addition, that it is necessary and sufficient
to guarantee convexity of the set $\V^b_{\cl Y}$ of $b$-convex functions.
A variant on the sufficiency was observed simultaneously and independently from us
in a different context by Sei (Lemma 1 of \cite{Sei09p}), who
was interested in the function $b(x,y) = -d^2_{S^n}(x,y)$, and
used it to give a convex parametrization of a family of statistical
densities he introduced on the round sphere $\X=\Y=S^n$.
\end{remark}

\section{Results concerning the principal-agent problem}

A mathematical concept of central relevance to us is encoded in the following definition.

\begin{definition}[$b$-convex]\label{D:b-convex}
A function $u:\cl{\X} \longmapsto \R$ 
is called $b$-convex if $u=(u^\bs)^b$,  where
\begin{equation}
\label{b-transform}
v^b(x) = \sup_{y \in \cl \Y} b(x,y) - v(y)
\quad {\rm and} \quad
u^\bs(y) = \sup_{x \in \cl \X} b(x,y) - u(x).
\end{equation}
In other words,  if $u$ is its own second $b$-transform, i.e.
a supremal convolution (or generalized Legendre transform) of some
function $v:\cl{\Y} \longmapsto \R \cup \{ +\infty\}$ with $b$.
The set of $b$-convex functions will be denoted by $\V^b_{\cl \Y}$.
Similarly, we define the set $\U^{\bs}_{\cl \X}$ of
$\bs$-convex functions to consist of those $v: \cl{\Y} \longmapsto \R$
satisfying $v=(v^b)^\bs$.
\end{definition}


Although some authors permit $b$-convex functions to take the value $+\infty$,
our hypothesis \Bzero\ ensures $b$-convex functions are Lipschitz continuous and thus
that the suprema defining their $b$-transforms are finitely attained.
Our first result is  the following.


\begin{theorem}[$b$-convex functions form a convex set]
\label{T:convex set U}
Assuming $b:\X \times \Y \longmapsto \R$ satisfies \Bzero--\Btwo,
hypothesis \Bthree\ becomes necessary and sufficient for the convexity of the
set $\V_{\cl\Y}^b$ of $b$-convex functions on $\cl\X$.
\end{theorem}

To understand the relevance of this theorem to economic
theory,
let us recall a mathematical formulation of the principal-agent problem
based on \cite{Carlier01} and  \cite{Rochet85} \cite{Rochet87}.
In this context,  each product $y \in \cl \Y$ costs the principal
$c(y)$ to manufacture,  and she is free to market this product to
the population $\cl\X$ of agents at any lower semicontinuous price $v(y)$
that she chooses.
She is aware that product $y$ has utility $b(x,y)$ to agent $x \in \cl \X$,
and that in response to any price menu $v(y)$ she proposes,
each agent will compute his indirect utility
\begin{equation}\label{indirect utility}
u(x) = v^b(x) := \max_{y \in \cl\Y} b(x,y) - v(y),
\end{equation}
and will choose to buy a product $y_{b,v}(x)$ which attains the maximum,
meaning $u(x)=b(x,y_{b,v}(x)) - v(y_{b,v}(x))$.
However,  let us assume that there is a distinguished point
$y_\emptyset \in \cl Y$ representing the null product, {which
the principal is compelled to offer to agents at zero profit,
\begin{equation}\label{null constraints}
v(y_\emptyset)= c(y_\emptyset), 
\end{equation}
either because both quantities vanish (representing the null transaction),
or because, as in \cite{ButtazzoCarlier09p}, there is a competing supplier
or regulator from whom the agents can obtain this product at price $c(y_\emptyset)$.}
In other words, $\ur(x) := b(x,y_\emptyset)-c(y_\emptyset)$
acts as the reservation utility of agent $x \in X$,
below which he will reject the principal's offers and decline to participate, whence $u \ge \ur$.
The map $y_{b,v}:\cl \X \longmapsto \cl \Y$ from agents to products they select
will not be continuous except possibly if the price menu $v$ is $\bs$-convex; when $y_{b,v}(x)$
depends continuously on $x\in\cl \X$
we say $v$ is {\em strictly} $\bs$-convex.

Knowing $b$, $c$ and a (Borel) probability measure $\mu$ on $\X$ --- representing
the relative frequency of different types of agents in the population ---
the principal's problem is to decide which lower semicontinuous
price menu $v:\cl \Y \longmapsto \R \cup \{+\infty\}$
maximizes her profits,  or equivalently, minimizes her net losses:
\begin{equation}\label{primitive principal's problem}
\int_{\X}
[ c(y_{b,v}(x)))- v(y_{b,v}(x))] d\mu(x).
\end{equation}
Note the integrand vanishes \eqref{null constraints}--\eqref{primitive principal's problem}
 for any agent $x$ who elects not to participate
(i.e., who chooses the null product $y_\emptyset \in \cl \Y$).


For absolutely continuous distributions of agents,
 --- or more generally if $\mu$ vanishes on Lipschitz hypersurfaces ---
it is known that the principal's losses
\eqref{primitive principal's problem} depend on $v$ only through the indirect utility $u=v^b$,
an observation which can be traced back to Mirrlees \cite{Mirrlees71} in one dimension
and Rochet \cite{Rochet85} more generally;  see also Carlier \cite{Carlier01}.
This indirect utility $u \ge \ur$ is $b$-convex, due to the well-known
identity $((v^b)^{b^*})^b=v^b$; see for instance Exercise 2.35 at page 87 of \cite{Villani03}.
Conversely, the principal can design any $b$-convex function
$u \ge \ur$ that she wishes simply by choosing price strategy $v= u^\bs$.
Thus,  as detailed below, the principal's problem can be reformulated as a
minimization problem \eqref{principal's reformulation} on the set
$\U_0:= \{u \in \V^b_{\cl \Y} \mid u \ge \ur\}$.
Under hypotheses \Bzero--\Bthree,  our Theorem \ref{T:convex set U}
shows the set $\V^b_{\cl \Y}$ of such utilities $u$ to be convex, in the usual sense.
This represents substantial progress, even though
the minimization problem (\ref{primitive principal's problem}) still depends nonlinearly
on $v=u^{\bs}$.  If, in addition,  the principal's cost $c(y)$ is a
$\bs$-convex function, then Proposition \ref{P:convexity}
and its corollary show her minimization problem
\eqref{primitive principal's problem} becomes a convex
functional of $u$ on $\U_0$,  so the principal's problem reduces to a convex program.
Necessary and sufficient conditions for a minimum can in principle then be expressed using
Kuhn-Tucker type conditions,  and numerical examples could be solved using standard algorithms.
However we do not do this here:  unless $\mu$ is taken to be a
finite combination of Dirac masses,  the infinite dimensionality of the convex set $\V^b_{\cl\Y}$
leads to functional analytic subtleties even for the bilinear preference function
$b(x,y) = \langle x, y \rangle$, which have only been resolved with partial success
by Rochet and Chon\'e in that case \cite{RochetChone98} \cite{Carlier02}.
If the $\bs$-convexity of $c(y)$ is {\em strict} however,
or if the preference function is positively cross-curved \BthreeU,
we shall show the principal's program has enough strict convexity to yield
unique optimal strategies for both the principal and the agents
in a sense made precise by Theorem~\ref{T:uniqueness}.  These optimal
strategies represent a Stackelberg (rather than a Nash) equilibrium,  in the sense that no party
has any incentive to change his or her strategies,  given that the
principal must commit to and declare her strategy before the agents select theirs.

Of course, it is of practical interest that the principal be able to anticipate
not only her optimal price menu $v: \cl \Y \longmapsto \R \cup \{+\infty\}$ ---
also known as the {\em equilibrium} prices ---  but the corresponding distribution of goods
which she will be called on to manufacture. This can be represented as a Borel probability measure
$\nu$ on $\cl Y$,  which we call the {\em optimal production measure}.  It quantifies
the relative frequency of goods to be produced, and is the image
of $\mu$ under the agents' best response function
$y_{b,v}:\cl \X \longmapsto \cl \Y$ to the principal's optimal strategy $v$.
This image $\nu = (y_{b,v})_\#\mu$  is a Borel probability measure on $\cl \Y$
known as the {\em push-forward} of $\mu$ by $y_{b,v}$, and is defined by the formula
\begin{equation}\label{push-forward}
\nu(W) := \mu[y_{b,v}^{-1}(W)]
\end{equation}
for each $W \subset \Y$.
Theorem \ref{T:uniqueness} asserts the optimal production measure
$\nu$ is unique and the optimal price menu $v$ is uniquely determined $\nu$-a.e.;
the same theorem gives a sharp lower bound for $v$ throughout $\cl\Y$.
If the convex domain $X_{y_\emptyset}$ is strictly convex and
the density of agents is Lipschitz continuous on $\X$,
Theorem \ref{T:Armstrong} goes on to assert that these prices will
be high enough to drive a positive fraction of agents out of the market,
extending Armstrong's desirability of exclusion  \cite{Armstrong96}
to a rich class of multidimensional models. Thus the
goods to be manufactured and their prices are uniquely determined at equilibrium, and the principal
can price the goods she prefers not to trade arbitrarily high but not arbitrarily low.
Theorem \ref{T:uniqueness} goes on to assert that the optimal strategy $y_{b,v}(x)$
is also uniquely determined for $\mu$-almost every agent $x$ by $b,c$ and $\mu$,
for each Borel probability measure $\mu$ on $\X$.
Apart from Theorem \ref{T:Armstrong},
these conclusions apply
to singular and discrete measures as well as to continuous measures $\mu$,
assuming the tie-breaking conventions of Remark~\ref{R:tie-breaker}
are adopted whenever $\mu$ fails to vanish on each Lipschitz hypersurface.

A number of examples of preference functions $b(x,y)$ which satisfy our
hypotheses are developed in
\cite{FigalliRifford08p} \cite{KimMcCann07p} \cite{KimMcCann08p} 
\cite{LeeLi09p} \cite{LeeMcCann09p} \cite{Loeper08p} \cite{MaTrudingerWang05} \cite{TrudingerWang07p}.
Here we mention just three:

\begin{example}
\label{E:linear}
For single dimensional type and allocation spaces $n=1$,
hypotheses \Bone--\Btwo\ are equivalent to asserting
that the preference function $b(x,y)$ be defined on a product of two intervals
where its cross-partial derivatives
$b_{xy}$ do not vanish.   Positive cross-curvature \BthreeU\ asserts that $D^2_{xy} b$ in turn satisfies a
Spence-Mirrlees condition, by having positive cross-partial derivatives:  $D^2_{xy}(D^2_{xy}b) > 0$.
\end{example}

\begin{example}
\label{E:bilinear}
The bilinear preference function $b(x,y) = x \cdot y$ of Armstrong, Rochet and Chon\'e
satisfies \Bzero--\Bthree\ provided only that $X,Y \subset \Rn$ are convex bodies.  In this
case 
$b$- and $\bs$-convexity coincide with
ordinary convexity.  Thus Theorem \ref{T:uniqueness} asserts that any strictly convex
manufacturing cost $c(y)$ leads to
unique optimal strategies for the principal and for $\mu$-almost every agent.
This uniqueness is well-known for absolutely continuous measures $d\mu \ll d$vol \cite{RochetChone98},
and Carlier and Lachand-Robert have {extended Mussa and Rosen's differentiability result
$u \in C^1(\cl X)$ to $n \ge 1$ in that case \cite{CarlierLachand-Robert01} \cite{MussaRosen78},}
but the uniqueness of optimal strategies
under the tie-breaking rules described in Remark \ref{R:tie-breaker}
may be new results when applied, for example, to discrete distributions $\mu$ concentrated on
finitely many agent types.
\end{example}

\begin{example}\label{E:convex product perturbation}
{Ma, Trudinger and Wang's} perturbation $b(x,y) = x \cdot y + F(x) G(y)$ of the bilinear
preference function is non-negatively cross-curved \Bthree\ provided
$F\in C^4\big(\cl X\big)$ and $G\in C^4\big(\cl Y\big)$ are
both convex 
\cite{MaTrudingerWang05} \cite{KimMcCann07p};
it is positively cross-curved if the convexity is strong,  meaning both
$F(x) - \epsilon |x|^2$ and $G(y)-\epsilon |y|^2$ remain convex for some $\epsilon>0$.
It satisfies \Bzero--\Bone\ provided $\sup_{x \in \X} |DF(x)| < 1$ and
$\sup_{y \in \Y} |DG(y)| < 1$,  and \Btwo\ if the convex domains $X$ and $Y \subset \Rn$ are
sufficiently convex,  meaning all principal curvatures of these domains are sufficiently large
at each boundary point \cite{MaTrudingerWang05}.  On the other hand,
$b(x,y) = x \cdot y + F(x) G(y)$ will violate \Bthree\ if $D^2F(x_0)>0$ holds but $D^2G(y_0) \ge 0$
fails at some $(x_0,y_0) \in \cl \X \times \cl \Y$.
\end{example}

In the next section we formulate the results mathematically.
Let us first highlight one implication of our results
concerning robustness of the phenomena observed by Rochet and Chon\'e.
Their bilinear function $b(x,y) = x \cdot y$ lies on the boundary of the set of
non-negatively cross-curved preference functions, since its cross-curvature
\eqref{MTW} vanishes identically.  Our results show non-negative cross-curvature
\Bthree\ to be a necessary and sufficient condition for the principal-agent
problem to be a convex program:  the feasible set $\V^b_{\cl\Y}$ becomes non-convex
otherwise,  and it is reasonable to expect that uniqueness of the solution among other
phenomena observed in \cite{RochetChone98} may be violated in that case.
In analogy with the discontinuities discovered by Loeper \cite{Loeper07p}, we
therefore conjecture that the bundling discovered by Rochet and Chon\'e
is robust with respect to perturbations of the bilinear preference function which
respect \Bzero--\Bthree, but not {generally}
with respect to perturbations violating~\Bthree.

\section{Mathematical formulation}


Any price menu $v:\cl \Y \longmapsto \R \cup\{\infty\}$ satisfies
\begin{equation}\label{vinequality}
v^b(x) + v(y) - b(x,y) \ge 0
\end{equation}
for all $(y,x) \in \cl \Y \times \cl \X$, according to definition \eqref{b-transform}.
Comparison with \eqref{indirect utility} makes it clear that
a (product, agent) pair produces equality in \eqref{vinequality} if and only if
selecting product $y$ is among the best responses of agent $x$ to this menu;
the set of such best-response pairs is denoted by $\partial^\bs v \subset \cl \Y \times \cl \X$;
see also \eqref{b-subdifferential}.
We think of this relation as giving a multivalued correspondence between
{products and agents}:
given price menu $v$ the set of agents (if any) willing to select product $y$
is denoted by $\partial^\bs v(y)$.
It turns out
$\partial^\bs v(y)$ is non-empty for all $y \in \cl\Y$ if and only if $v$ is $\bs$-convex.
Thus $b^*$-convexity of $v$ --- or of $c$ --- means precisely that each product is priced low
enough to be included among the best responses of some agent or limiting agent type $x \in \cl \X$.
As we shall see in Remark \ref{R:tie-breaker},
assuming $b^*$-convexity of $v$ costs little or no generality;
however, the $b^*$-convexity of $c$ is a real restriction ---
but plausible when the product types $\Y \subset \Rn$
represent mixtures {\dummy (weighted combinations of pure products)
which the principal could alternately choose to purchase separately
and then bundle together; this becomes natural in the context of von-Neumann and
Morgenstern preference functions \cite{vonNeumannMorgenstern44} like the one 
used by Rochet and Chon\'e \cite{RochetChone98}.}

Let $\dom Du \subset \cl\X$ denote the set where $u$ is differentiable.
If $y$ is among the best responses of agent $x \in \dom D v^b$ to price menu $v$,
the equality in \eqref{vinequality} implies
\begin{equation}\label{marginal utility}
D v^b(x) = D_x b(x,y).
\end{equation}
In other words $y=y_b(x,Dv^b(x))$, where $y_b$ is defined as follows:

\begin{definition}\label{D:b-exp}
For each $x \in \cl{\X}$ and $q \in \cl Y_{x}$, let us define $y_b(x,q)$ to be the unique solution to
\begin{equation}\label{b-exp}
D_x b(x,y_b(x,q)) = q
\end{equation}
guaranteed by \Bone. The map $y_b$ (which is defined on a subset of the cotangent
bundle $T^* \cl\X$ and takes values in $\cl Y$) has also been called the $b$-exponential map  \cite{Loeper07p},
and denoted by $y_b(x,q) = \bexp_x q$.
\end{definition}

The fact that the best response function takes the form $y=y_b(x,Dv^b(x))$,
and that $\dom D v^b$ exhausts $\cl X$ except for a countable number of Lipschitz hypersurfaces,
are among the key observations exploited in
 \cite{Gangbo95} \cite{Levin99}
following special cases worked out in
\cite{Armstrong96} \cite{Brenier91}  \cite{Caffarelli96} \cite{GangboMcCann96}
\cite{McCann95} \cite{RochetChone98}.  Indeed, $v^b$ is well-known
to be a $b$-convex function.  It is therefore Lipschitz and semiconvex, satisfying the bounds
\begin{equation}
\label{eq:lip semiconv}
|D v^b| \le \|c\|_{C^1\big(\X \times \Y\big)},\qquad D^2 v^b \ge - \|c\|_{C^2\big(\X \times \Y\big)} \qquad \text{inside }\cl X.
\end{equation}
The second equation above holds in the distributional sense, and implies the differentiability of $v^b$ outside a
countable number of Lipschitz hypersurfaces.

Assuming $\mu$ assigns zero mass to each Lipschitz hypersurface (and so also to a countable number of them),
the results just summarized
allow 
the principal's problem \eqref{primitive principal's problem} to be re-expressed
in the form
$\min\{L(u) \mid u \in \U_0\}$, where the principal's net losses are given by

\begin{equation}\label{principal's reformulation}
L(u) := \int_{\X} [u(x) + c(y_b(x,Du(x))) - b(x,y_b(x,Du(x)))] d\mu(x)
\end{equation}
as is by now well-known \cite{Carlier01}.
Here $\U_0 = \{ u \in \V^b_{\cl Y} \mid u \ge \ur\}$
denotes the set of $b$-convex functions on $\cl X$ dominating the reservation utility
$\ur(x) = b(x,y_\emptyset){-c(y_\emptyset)}$,  and 
the equality produced
in \eqref{vinequality} by the response $y_{b,v}(x) = y_b(x,D v^b(x))$
for $\mu$-a.e.\ $x$ has been exploited. Our hypothesis
on the distribution of agent types holds a fortiori whenever $\mu$ is absolutely continuous
with respect to Lebesgue measure in coordinates on $\X$.
If no such hypothesis is satisfied,
the reformulation \eqref{principal's reformulation}
of the principal's net losses
may not be well-defined, unless we extend the definition of $Du(x)$ 
to all of $\X$ by making a measurable selection from the relation
$$
\partial u(x) := \{ q \in \R^n \mid u(z) \ge u(x) + q \cdot (z-x) + o(|z-x|) \quad\forall\ z \in X\}
$$
consistent with the following tie-breaking rule, analogous to one adopted, e.g., by
Buttazzo and Carlier 
in a different but related context~\cite{ButtazzoCarlier09p}:

\begin{remark}\label{R:tie-breaker}[Tie-breaking rules for singular measures]
When an agent $x$ remains indifferent between two or more products,
it is convenient to reduce the ambiguity in the definition of his
best response by insisting that $y_{b,v}(x)$ 
be chosen to maximize the principal's profit $v(y)-c(y)$,
among those products $y$ which maximize \eqref{indirect utility}.
We retain the result $y_{b,v}(x) = y_b(x,Dv^b(x))$ by a corresponding selection
$Dv^b(x) \in \partial v^b(x)$.
This convention costs no generality when the distribution $\mu$ of agent types
vanishes on Lipschitz hypersurfaces in $\X$, since $u=v^b$
is then differentiable $\mu$-a.e.; in the remaining cases it may be
justified by assuming the principal has sufficient powers of persuasion to sway an
agent's choice to her own advantage
whenever some indifference would otherwise persist
between his preferred products \cite{Mirrlees71}.
After adopting this convention,  it costs the
principal none of her profits to restrict her choice of strategies to $\bs$-convex
price menus $v=(v^b)^\bs$, a second convention we also choose to adopt whenever
$\mu$ fails to vanish on each Lipschitz hypersurface.
\end{remark}

The relevance of Theorem \ref{T:convex set U} to the principal-agent problem
should now be clear:  it guarantees convexity of the feasible set $\U_0$ in
\eqref{principal's reformulation}.  Our next proposition addresses the convexity
properties of the principal's objective functional.  Should convexity of this objective be strict,
then the 
best response $y_{b,v}(x)$ 
selected by the tie-breaking rule above becomes unique
--- which it need not be otherwise.

\begin{proposition}[Convexity of the principal's objective]
\label{P:convexity}
If $b\in C^4\big(\cl\X \times \cl\Y\big)$ satisfies \Bzero--\Bthree\ and
$c:\cl \Y \longmapsto \R$ is $\bs$-convex,  then
for each $x \in \cl X$, definition \eqref{b-exp} makes
$a(q):= c(y_b(x,q)) - b(x,y_b(x,q))$
a convex function of $q$
on the convex set $\cl\Y_x := D_x b(x,\cl\Y) \subset \R^n$.
The convexity of $a(q)$ is strict if either $c$ is strictly $\bs$-convex ---
meaning the {\em efficient} allocation $y_{b,c}:\cl \X \longmapsto \cl \Y$ is continuous ---
or alternately if
\begin{equation}\label{BthreeS}
q \in \cl\Y_x \longmapsto b(x_0,y_b(x,q)) - b(x,y_b(x,q))
\end{equation}
is a strictly convex function of $q$ for each $x,x_0 \in \cl\X$.
If the preference function
$b$ is positively cross-curved 
on $\cl \X \times \cl \Y$,  then convexity of $a(q)$
is {\em strong} (meaning $a(q) - \epsilon|q|^2/2$ remains convex on $\cl\Y_x$ for
some $\epsilon>0$).
\end{proposition}


Strict convexity of \eqref{BthreeS} may subsequently be denoted by \BthreeS.
As an immediate corollary to Theorem \ref{T:convex set U} and Proposition \ref{P:convexity},
we have convexity of the principal's optimization problem.

\begin{corollary}[Convexity of the principal's minimization]
\label{C:convex program}
Let the distribution of agent types be given by a Borel probability measure $\mu$
on $X \subset \Rn$.  Unless $\mu$ vanishes on all Lipschitz hypersurfaces, adopt
the tie-breaking conventions 
of Remark~\ref{R:tie-breaker}.
If the preference $b(x,y)$ of agent $x \in \cl\X$ for product
$y\in \cl\Y$ satisfies \Bzero--\Bthree\ and the principal's
manufacturing cost $c:\cl \Y \longmapsto \R$ is $\bs$-convex,
then the principal's problem \eqref{principal's reformulation}
becomes a convex minimization over the convex set $\U_0$.
\end{corollary}

As a consequence, we obtain criteria guaranteeing uniqueness of the principal's
best strategy.

\begin{theorem}[Criteria for uniqueness of optimal strategies]
\label{T:uniqueness}
Assume the notation and hypotheses of Corollary \ref{C:convex program}.
Suppose, in addition, that the manufacturing cost $c$ is strictly $\bs$-convex, or that
the preference function $b$ is positively cross-curved \BthreeU, or that $b$
satisfies \BthreeS, as in \eqref{BthreeS}.  Then the equilibrium response of $\mu$-almost every agent
is uniquely determined,  as is the optimal measure $\nu$ from
\eqref{push-forward};
(always assuming the tie-breaking conventions of Remark \ref{R:tie-breaker} to be in effect
 if $\mu$ does not vanish on each Lipschitz hypersurface).
Moreover, the principal has two optimal strategies $u_\pm \in \U_0$
which coincide at least $\mu$-almost everywhere,  and sandwich
all other optimal strategies $u \in \U_0$ between them:
$u_- \le u \le u_+$ on $\cl \X$.  Finally,
a lower semicontinuous $v:\cl\Y \longmapsto \R \cup \{+\infty\}$ is an
optimal price menu if and only if $v \ge u_+^\bs$ throughout $\cl \Y$, with
equality holding $\nu$-almost everywhere.
\end{theorem}

This theorem gives hypothesis which guarantee --- even for discrete measures $\mu$ corresponding
to finitely many agent types --- that the solution to the principal's problem is unique in the
sense that optimality determines how many of each type of product the principal should manufacture,
what price she should charge for each of them, and which product will be selected
by almost every agent.  A lower bound is specified on the price of each product which
she does not wish to produce, to ensure that it does not tempt any agent.
When $\mu$ vanishes on Lipschitz hypersurfaces,  this solution represents
the only Stackelberg equilibrium balancing the interests of the principal with those of the agents;
for more singular $\mu$,  it is possible that other Stackelberg equilibria exist, but if so they
violate the restrictions imposed on the behaviour of the principal and the agents in
Remark~\ref{R:tie-breaker}.

{The uniqueness theorem has as its corollary the following stability result
concerning optimal strategies.
Recall that a sequence $\{\mu_i\}_{i=1}^\infty$ of Borel probability measures on a compact set $\cl \X \subset \Rn$
is said to converge {\em weakly-$*$} to $\mu_\infty$ if
\begin{equation}\label{weak convergence}
\int_{\cl \X} g(x) d\mu_\infty(x) = \lim_{i \to \infty} \int_{\cl \X} g(x) d\mu_i(x)
\end{equation}
for each continuous test function $g:\cl \X \longmapsto \R$.
This notion of convergence makes the Borel probability measures $\Prob\big(\cl X\big)$ on
$\cl\X$ into a compact set, as a consequence of the Riesz-Markov and Banach-Alaoglu theorems.

\begin{corollary}[Stability of optimal strategies]\label{C:stability}
For each $i \in \N \cup \{ \infty\}$, let the triple $( b_i, c_i,  \mu_i)$
consist of a preference function $b_i :  \cl \X \times \cl \Y \longmapsto \R$,
manufacturing cost $c_i:\cl \Y \longmapsto \R$,
and a distribution of agent types $\mu_i $ on $\X$ satisfying the hypotheses of
Theorem~\ref{T:uniqueness}.  Let $u_i:\cl X \longmapsto \R$ denote a $b_i$-convex utility function
minimizing the losses of a principal faced with data $( b_i, c_i,  \mu_i)$.
Suppose that $b_i \to b_\infty$ in $C^2\big(\cl \X \times \cl \Y\big)$, $c_i \to c_\infty$ uniformly on
$\cl \Y$, and $\mu_i \rightharpoonup \mu_\infty$ weakly-$*$ as $i\to \infty$.
Assume finally that
$\mu_\infty$ vanishes on all Lipschitz hypersurfaces.  For $\mu_\infty$-a.e.\ agent $x \in \X$, the
product
$G_i(x) := y_{b_i}(x,D u_i(x))$ selected then converges to $G_\infty(x)$.
The 
optimal measures $\nu_i := (G_i)_\# \mu_i$ converge weakly-$*$
to $\nu_\infty$ as $i \to \infty$.  And the principal's
strategies converge uniformly in the sense that
$\lim_{i\to \infty } \|u_i - u_\i\|_{L^\infty(X,d\mu_\infty)}=0$.

\end{corollary} }


Finally as evidence for the robustness of bunching phenomena displayed by our models,
we show the desirability of exclusion phenomenon found by Armstrong
for preference functions $b(x,y) = \sum_{i=1}^n x_i b_i(y)$ which are linear in $x$
--- or more generally homogeneous \cite{Armstrong96} ---
extends to the full range of non-negatively cross-curved models.
We assume strict convexity on the domain $X_{y_\emptyset} := D_y b(X,y_\emptyset)$
(see Remark \ref{R:facets}), and
that the distribution of agent types $d\mu(x)=f(x)dx$
has a {\dummy Sobolev density --- denoted $f \in W^{1,1}\big(\cl X\big)$ and meaning both
the function and its distributional derivative $Df$ are given by Lebesgue integrable
densities.}  This is satisfied a fortiori if $f$ is Lipschitz or continuously
differentiable (as Armstrong assumed). The exclusion phenomenon is of interest,
since it confirms that a positive fraction of customers
must be excluded from participation at equilibrium, thus ensuring elasticity of demand.

\begin{theorem}[The desirability of exclusion]
\label{T:Armstrong} Let the distribution  $d\mu(x)=f(x)dx$ of
agent types be given by a density $f \in W^{1,1}$
on $\cl \X \subset \Rn$. Assume that the preference $b(x,y)$ of agent
$x \in \cl\X$ for product $y\in \cl\Y$ satisfies \Bzero--\Bthree\
and the principal's manufacturing cost $c:\cl \Y \longmapsto \R$
is $\bs$-convex. 
Suppose further
that the convex domain $X_{y_\emptyset}=D_xb(X,y_\emptyset)$ has no $n-1$ dimensional
facets in its boundary. Then any minimizer $u \in \U_0$ of the principal's losses
\eqref{principal's reformulation} coincides with the reservation
utility on a set $U_0 := \{x \in \cl \X \mid
u(x)=b(x,y_\emptyset)-c(y_\emptyset)\}$ whose interior contains a positive
fraction of the agents. Such agents select the null product
$y_\emptyset$.
\end{theorem}


\begin{remark}[{\dummy Facets and} exclusion in different dimensions]
\label{R:facets}
A convex domain $X \subset \Rn$ fails to be strictly convex if it has line segments in its boundary.
These segments belong to facets of dimension $1$ or higher,  up to $n-1$ if the domain has a flat side
(meaning a positive fraction of its boundary coincides with a supporting hyperplane).  Thus strict
convexity of $X_{y_\emptyset}$ is sufficient for the hypothesis of the preceding theorem to be
satisfied --- except in dimension $n=1$.  In a single dimension,  every convex domain
$X \subset \R$ is an interval --- hence strictly convex --- whose endpoints form zero-dimensional
facets.  Thus Theorem \ref{T:Armstrong} is vacuous in dimension $n=1$,  which is consistent with
Armstrong's observation the necessity of exclusion is a hallmark of higher dimensions $n \ge 2$.
More recently,
Deneckere and Severinov \cite{DeneckereSeverinov09p} have argued that necessity of exclusion
is specific to the case in which the dimensions $m$ and $n$ of agent and product types coincide.
When $(m,n)=(2,1)$ they give necessary and sufficient conditions for the desirability of exclusion,
yielding a result quite different from ours in that exclusion turns out to be more frequently the
exception than the rule.
\end{remark}

{
\section{Discussion, extension, and conclusions}

The role of private information in determining market value
has a privileged place in economic theory,  acknowledged by the award
of the Nobel Memorial Prize in Economic Sciences to Mirrlees and Vickrey in 1996,
and to Akerlof, Spence and Stiglitz in 2001.  This phenomenon
has been deeply explored in the principal-agent framework,
where a single seller (or single buyer) transacts
business with a collection of anonymous agents.  In this context,
the private (asymmetric) information takes the form of a characteristic $x \in \X$
peculiar to each individual buyer which determines his preference
$b(x,y)$ for different products $y \in \Y$ 
offered by the principal;
$x$ remains concealed from the principal by anonymity of the buyer ---
at least until a purchase is made.  Knowing only the preference function
$b(x,y)$,  the statistical distribution $d\mu(x)$ of buyer types,
and her own manufacturing costs $c(y)$,
the principal's goal is to fix a price menu for different products which
maximizes her profits.

Many studies involving finite spaces of agent and product types $\X$ and $\Y$
have been carried out, including Spence's initial work on labour market signalling \cite{Spence73} and
Stiglitz paper with Rothschild on insurance \cite{RothschildStiglitz76}.  However
for a principal who transacts business with a one-dimensional continuum of agents $\X \subset \R$,
the problem was solved in Mirrlees' celebrated work on optimal taxation \cite{Mirrlees71},
and in Spence's study \cite{Spence74},  assuming the
contract types $y \in \Y \subset \R$ are also parameterized by a single real variable.
(For Mirrlees, $y \in \R$ represented the amount of labour an individual chooses to do
 facing a given tax schedule,
 while for Spence it represented the amount of education he chooses to acquire
 facing a given range of employment possibilities,
 $x \in \R$ being his intrinsic ability in both cases).
In the context of nonlinear pricing discussed above,
the one-dimensional model was studied by Mussa and Rosen \cite{MussaRosen78}.
The challenge of resolving the multidimensional version $\X,\Y \subset \Rn$
of this archetypal problem in microeconomic theory has been highlighted by
many authors \cite{McAfeeMcMillan88} \cite{Wilson91} \cite{Mirrlees96}
\cite{RochetStole03} \cite{Basov05}.
When only one side of the market displays multidimensional types,
analyses have been carried out by
Mirman and Sibley~\cite{MirmanSibley80}, Roberts \cite{Roberts79} and Spence \cite{Spence80},
who allow multidimensional products,  and by Laffont, Maskin and Rochet \cite{LaffontMaskinRochet87},
Araujo, Gottlieb and Moreira \cite{AraujoGottliebMoreira07},
and Deneckere and Severinov \cite{DeneckereSeverinov09p}
who model two-dimensional agents choosing from a one-dimensional product line.
When both sides of the market display multidimensional types,  existence of an
equilibrium has been established by Monteiro and Page \cite{MonteiroPage98}
and by Carlier \cite{Carlier01}, who employed a variational
formulation; see also the control-theoretic approach of Basov \cite{Basov01} \cite{Basov05}.
However, non-convexities have rendered the behaviour of this optimization problem
largely intractable \cite{GuesnerieLaffont78} ---
unless the preference function $b(x,y) = x \cdot G(y)$ is assumed
to depend linearly on agent type \cite{Wilson93} \cite{Armstrong96} \cite{RochetChone98}.
Moreover,  the presence of convexity typically depends on a correct choice of coordinates,
so is not always easy to discern.
The present study treats general Borel probability measures $\mu$ on $\X \subset \Rn$,
and provides a unified framework for dealing with discrete and continuous type spaces,
by invoking the tie-breaking rules of Remark \ref{R:tie-breaker} in case $\mu$
is discrete.  Assuming $b^*$-convexity of $c$, we
consider preferences linear in price
\eqref{indirect utility} (sometimes called quasilinear),
which satisfy a generalized Spence-Mirrlees single
crossing condition \Bzero-\Bone\ and appropriate convexity conditions on its domain
\Btwo, and we identify a criterion \Bthree\ equivalent
to convexity of the principal's optimization problem (Theorem \ref{T:convex set U}).
This criterion is a strengthening of Ma, Trudinger and Wang's necessary \cite{Loeper07p}
and sufficient \cite{MaTrudingerWang05} \cite{TrudingerWang07p} condition for
continuity of optimal mappings. Like all of our hypotheses,
it is independent of the choice of parameterization of agent and/or product types
--- as emphasized in \cite{KimMcCann07p}.
We believe the resulting convexity is a fundamental property which will eventually enable
a more complete theoretical and computational analysis of the multidimensional principal-agent problem,
and we indicate some examples of preference
functions which satisfy it in Examples \ref{E:linear}--\ref{E:convex product perturbation};
the bilinear example $b(x,y)= x \cdot y$ of Rochet and Chon\'e lies on the boundary of such
preference functions. If either the cross-curvature inequality
\Bthree\ holds strictly or the $b^*$-convexity of $c(y)$ is strict --- meaning
the {\em efficient} solution $y_{b,c}(x)$ depends continuously on $x \in\cl \X$ ---
we go on to derive uniqueness and stability of optimal strategies
(Theorem \ref{T:uniqueness} and its corollary).  Under mild additional hypotheses
we confirm that a positive fraction of agents must be priced out of the market
{\dummy when the type spaces are multidimensional} (Theorem~\ref{T:Armstrong}).  We conjecture that
non-negative cross-curvature \Bthree\ is likely to be necessary and sufficient for robustness
of Armstrong's desirability of exclusion \cite{Armstrong96} and the other bunching phenomena
observed by Rochet and Chon\'e \cite{RochetChone98}.

\begin{remark}[Maximizing social welfare under profitability constraints]\label{R:SWM under profitability constraint}
Before concluding this paper,  let us briefly mention an important class of related
models to which the same considerations apply:  namely,  the problem of maximizing the
expected welfare of the agents under a profitability constraint
 on the principal.  Such a model has been used by Roberts \cite{Roberts79}
to study energy pricing by a public utility, and explored by Spence \cite{Spence80} and
Monteiro and Page \cite{MonteiroPage98} in other contexts.  Suppose the
welfare of agent $x \in \X$ is given by a concave function $w(u(x))$ of his
indirect utility \eqref{indirect utility}.  Introducing a Lagrange multiplier
{\dummy $\lambda$ for the profitability constraint $L(u) \le 0$,  the problem of maximizing
the net social welfare over all agents becomes equivalent to the maximization
$$
W(\lambda) := \max_{u \in \U_0} -\lambda L(u) + \int_\X w(u(x)) d\mu(x)
$$
}for some choice of $\lambda \ge 0$.
Assuming \Bzero--\Bthree, and $b^*$-convexity of $c$,
for each $\lambda \ge 0$ this amounts to a concave maximization on a convex set,
as a consequence of Theorem \ref{T:convex set U}, Proposition \ref{P:convexity}
and the concavity of $w$.
Theorem \ref{T:uniqueness} and its corollary give hypotheses which guarantee
uniqueness and stability of its solution $u_\lambda$;  if the concavity of $w$ is strict,
we obtain uniqueness $\mu$-a.e.\ of $u_\lambda$ more directly under the weaker hypotheses of
Corollary \ref{C:convex program}.
{\dummy Either way,  once the uniqueness of $u_\lambda$ has been established,
standard arguments in the calculus of variations show the convex function $W(\lambda)$ to be
continuously differentiable, and that each value of its derivative
$W'(\lambda) = -L(u_\lambda)$ corresponds to a possibly degenerate interval
$\lambda \in [\lambda_1,\lambda_2]$ on which $u_\lambda$ is constant;  see
e.g.~Corollary 2.11 of \cite{CaffarelliMcCann99}. 
Uniqueness of a social welfare maximizing strategy subject to any
budget constraint in the range $]L(u_0),L(u_\infty)[$ is therefore established;
this range contains the vanishing budget constraint as long as $L(u_0)>0>L(u_\infty)$;
here $u_0$ represents the unconstrained maximizer whereas $u_\infty \in \U_0$ minimizes
the principal's losses \eqref{principal's reformulation}.}
All of our results --- except for the desirability of exclusion (Theorem \ref{T:Armstrong})
--- extend immediately to this new setting.  This sole exception is in accord with
the intuition that it need not be necessary to exclude any potential buyers
if one aims to maximize social welfare instead of the monopolist's profits.
\end{remark}
}

\section{Proofs}

Let us recall a characterization of non-negative cross-curvature
from Theorem 2.11 of \cite{KimMcCann08p}, inspired by Loeper's
characterization \cite{Loeper07p} of \AthreeW.
We recall its proof partly for the sake of completeness, but also
to extract a criterion for strong convexity.

\begin{lemma}[Characterizing non-negative cross-curvature \cite{KimMcCann08p}]
\label{L:t-convex DASM}
A preference function $b$ satisfying \Bzero--\Btwo\ is non-negatively cross-curved \Bthree\
if and only if for each $x,x_1 \in \cl \X$
\begin{equation}\label{convex difference}
q \in \cl \Y_x \longmapsto b(x_1,y_b(x,q)) - b(x,y_b(x,q))
\end{equation}
is a convex function.  If the preference function is positively cross-curved, 
then \eqref{convex difference} will be strongly convex (meaning its Hessian will be positive definite).
\end{lemma}

\begin{proof}
Fix $x,x_1 \in \cl\X$ and
set $q_t := (1-t)q_0 + t q_1$ and $f(\cdot,t):= b(\cdot,y_b(x,q_t)) - b(x,y_b(x,q_t))$
for $t\in[0,1]$.  Given $t_0 \in [0,1]$,  use \Bone--\Btwo\ to define the curve
$s \in [0,1] \longmapsto x_s \in \cl\X$ for which
\begin{equation}\label{s-segment}
D_x b(x_s, y_{t_0}) = (1-s) D_x b(x, y_{t_0}) + s D_x b(x_1, y_{t_0}),
\end{equation}
and set $g(s) = \frac{\partial^2 f}{\partial t^2}(x_s,t_0)$.
The convexity of \eqref{convex difference} will be verified by checking
$g(1) \ge 0$.  Let us start by observing $s \in [0,1] \longmapsto g(s)$
is a convex function,  as a consequence of property \Bthree\ and \eqref{s-segment};
(according to Lemma 4.5 of \cite{KimMcCann07p}, inequality \eqref{MTW} follows from \Bthree\
 whenever either of the two curves $s \in [0,1] \longmapsto (D_y b(x(s),y(0))$
 or $s \in [0,1] \longmapsto D_x b(x(0),y(s)))$ is a line segment).
We next claim $g(s)$ is minimized at $s=0$,  since
$$
g'(s) = \frac{\partial^2}{\partial t^2}\bigg|_{t=t_0} \langle D_x b(x_s,y_b(x_0,(1-t)q_0 + tq_1)), \dot x_s\rangle
$$
vanishes at $s=0$,  by the definition \eqref{b-exp} of $y_b$.
Thus $g(1) \ge g(0) =0$, establishing the convexity of \eqref{convex difference}.
If $b$ is positively cross-curved, then $g''(s)>0$ and the desired strong
convexity follows from $g(1)> g(0)=0$.

Conversely,  if the convexity of \eqref{convex difference} fails
we can find $x_1 \in \cl X$ and $s_0,t_0 \in [0,1]$ for which the construction above yields
$g''(s_0)<0$. In view of Lemma 4.5 of \cite{KimMcCann07p}, this provides a contradiction to
\eqref{MTW}.
\end{proof}

We shall also need to recall two basic facts about $b$-convex functions
from e.g.\ \cite{GangboMcCann96}:  any supremum of $b$-convex functions
is again $b$-convex,  unless it is identically infinite; and for each
$y \in \cl \Y$ and $\lambda \in \R$, the function
\begin{equation}\label{mountain}
x \in \cl \X \longmapsto b(x,y) - \lambda
\end{equation}
is $b$-convex.  Functions of the form either $y \in \cl \Y \longmapsto b(x,y) - \lambda$
or \eqref{mountain} are sometimes called {\em mountains} below.

\begin{proof}[Proof of Proposition \ref{P:convexity}]
The $\bs$-convexity of the manufacturing cost $c=(c^b)^\bs$ asserts
$$
c(y) = \sup_{x \in \cl \X} b(x,y) - c^b(x)
$$
is a supremum of mountains, whence
$$
a(q) := c(y_b(x,q)) - b(x,y_b(x,q)) = \sup_{x_0 \in \cl \X} b(x_0,y_b(x,q))-b(x,y_b(x,q)) - c^b(x_0)
$$
for all $x \in \cl X$ and $q \in \cl\Y_x$. According to Lemma \ref{L:t-convex DASM},
we have just expressed $a(q)$ as a supremum of convex functions, thus establishing
convexity of $a(q)$.  The functions under the supremum are strictly convex if \eqref{BthreeS} holds,
and strongly convex if $b$ is positively cross-curved,
thus establishing the strict or strong convexity of $a(q)$ under the respective hypotheses
\BthreeS\ and \BthreeU.

The remainder of the proof 
will be devoted to deducing strict convexity of $a(q)$ from strict
$\bs$-convexity of $c(y)$ assuming only \Bthree.
Recall that strict $\bs$-convexity was defined by continuity of the
agents' responses $y_{b,c}:\cl\X \longmapsto \cl \Y$
to the principal's manufacturing costs (as opposed to the prices
the principal would prefer to select).
Fix $x \in \cl \X$ and use the $C^3$ change of variables
$q \in \cl \Y_x \longmapsto y_b(x,q) \in \cl \Y$
to define $\tb(\cdot,q):= b(\cdot,y_b(x,q)) - b(x,y_b(x,q))$  and
$\tc (q) := c(y_b(x,q))-b(x,y_b(x,q)) = a(q)$.  As in \cite{FigalliKimMcCann},
it is easy to deduce that $\tb$ satisfies the same hypotheses \Bzero--\Bthree\ on
$\cl \X \times \cl \Y_x$ as the original preference function 
--- except for the fact that $\tb \in C^3$ whereas $b \in C^4$.  For the reasons
explained in \cite{FigalliKimMcCann} this discrepancy shall not trouble us here: we
still have continuous fourth derivatives of $\tb$ as long as at least one of the
four derivatives is with respect to a variable in $\cl \X$, and at most three
derivatives are with respect to variables in $\cl \Y_x$.
Note also that $\tc^\tb = c^b$
and the continuity of the agents' responses $y_{\tb,\tc}$ in the new variables follows from their
presumed continuity in the original variables, since $y_{\tb,\tc}(\cdot) = D_x c(x,y_{b,c}(\cdot))$.

The advantage of the new variables is that for each $x_0 \in \cl X$, the
mountain $q \in \cl \Y_x \longmapsto \tb(x_0,q)$ is a convex function,  according to
Lemma \ref{L:t-convex DASM}; (alternately, Theorem 4.3 of \cite{FigalliKimMcCann}).
To produce a contradiction, assume convexity of $\tc(q)$ fails to be strict,  so there is a segment
$t \in [0,1] \longmapsto q_t  \in  \cl \Y_x$ given by $q_t = (1-t)q_0 + tq_1$
along which $\tc$ is affine with the same slope $p \in \partial \tc(q_t)$ for each $t \in [0,1]$.
In fact, the compact convex set $\partial\tc(q_t)$ is independent of $t\in ]0,1[$,
so taking $p$ to be an extreme point of $\partial\tc(q_t)$ allows us to find a
sequence $q_{t,k} \in \Y_x \cap \dom D\tc$ converging to $q_t$
such that $p = \lim_{k \to \infty} D \tc(q_{t,k})$, by Theorem 25.6 of
Rockafellar \cite{Rockafellar72}.
On the other hand, $\bs$-convexity implies
$\tc(q)$ is a supremum of mountains: thus
to each $t\in [0,1]$ and integer $k$ corresponds some $x_{t,k} \in \cl \X$ such that
$(x_{t,k},q_{t,k}) \in \partial^\tbs \tc$, meaning
\begin{equation}\label{supporting}
\tc(q) \ge \tb(x_{t,k},q) - \tb(x_{t,k},q_{t,k}) + \tc(q_{t,k})
\end{equation}
for all $q \in \cl \Y_x$.  Since $q_{t,k} \in \dom D\tc$, saturation of this bound at $q_{t,k}$ implies
$D\tc(q_{t,k})= D_q \tb(x_{t,k},q_{t,k})$. Compactness of $\cl \X$ allows us to extract a subsequential
limit $(x_{t,k},q_{t,k}) \to (x_t,q_t) \in \partial^\tbs \tc$ satisfying $p=D_q \tb(x_t,q_t)$.
This first order condition shows the curve $t \in [0,1] \longmapsto x_t \in \cl \X$ to be
differentiable, with derivative
\begin{equation}\label{x-velocity}
\dot x_t = -D^2_{qx} \tb(x_t,q_t)^{-1}D^2_{qq} \tb(x_t,q_t) \dot q_t,
\end{equation}
by the implicit function theorem and \Bone.
On the other hand,  both $\tc(\cdot)$ and $\tb(x_t,\cdot)$ are convex functions of
$q \in \cl \Y_x$ in \eqref{supporting}, so both must be affine along the segment $q_t$.
This implies $\dot q_t = q_1-q_0$ is a zero eigenvector of $D^2_{qq} \tb (x_t,q_t)$,
which in turn implies $x_t = const$ from \eqref{x-velocity}.  On the other hand, the
efficient response $q_t = y_{\tb,\tc}(x_t)$ of agent $x_t$ to price menu $\tc$ is not constant,
since the endpoints $q_0 \ne q_1$ of the segment are distinct.
This produces the desired contradiction and establishes strict convexity of $\tc$.
\end{proof}

Combining Proposition \ref{P:convexity} with the following standard lemma
will allow us to establish our necessary and sufficient criteria for convexity of
the feasible set $\U_0$.

\begin{lemma}[Identification of supporting mountains]
\label{lemma:tangent}
 Let $u$ be a $b$-convex function on $\cl{\X}$. Assume $u$ is differentiable at $x_0\in \X$ and
 $D_x u(x_0) = D_x b(x_0, y)$ for some $y \in \cl{\Y}$.
 Then, $u(x) \ge m(x)$ for all $x \in \X$, where
 $m (\cdot) = b(\cdot, y) -b(x_0, y) + u(x_0)$.
\end{lemma}

\begin{proof}
By $b$-convexity of $u$, there exists $y_0 \in \cl{\Y}$ such that
$u(x_0) = b(x_0, y_0) - u^{\bar b} (y_0)$ and also
$u(x) \ge b(x, y_0) -u^{\bar b}(y_0)$ for all $x \in \X$.
Since $u$ is differentiable at $x_0$, this implies $D_x u(x_0) = D_x b(x_0, y_0)$.
By the assumption \Bone, we conclude $y=y_0$. This completes the proof since
$m(\cdot) = b(\cdot, y_0) - u^{\bar b} (y_0)$.
\end{proof}

\begin{proof}[Proof of Theorem \ref{T:convex set U}]
Let us first show the sufficiency.
It is enough to show that for any two $b$-convex functions $u_0$ and $u_1$, the linear
combination $u_t:=(1-t)u_0 +tu_1$ is again $b$-convex, for each $0\le t \le 1$.
Fix $x_0 \in \cl{\X}$. Since $b$-convex functions are defined as suprema of mountains,
there exist $y_0$, $y_1 \in \cl{\Y}$ such that
$$
m^{x_0}_i (\cdot):=b(\cdot, y_i) -b(x_0 , y_i)    , \ \ \ i=0,1,$$ satisfy
$u_i(x) \ge m^{x_0}_i(x) +u_i(x_0)$ for all $x \in \overline{\X}$. Clearly
equality holds when $x=x_0$.
Let us consider the function
$$m^{x_0}_t(\cdot) = b(\cdot, y_t) -b(x _0, y_t) ,
$$
where $y_t$ defines a line segment
$$t \in [0,1] \longmapsto D_x b(x_0, y_t) = (1-t) D_x b(x_0, y_0) + t \, D _x b(x_0, y_1) \in \R^n.
$$
Note that (i) $m_t^{x_0}(x_0)=0$. We claim that (ii)  $u_t(\cdot) \ge m_t^{x_0}(\cdot) + u_t (x_0)$.
Notice that
$$
u_t(\cdot) \ge (1-t) m^{x_0}_0(\cdot) +t \, m^{x_0}_1(\cdot) + u_t (x_0) .
$$
Thus the claim follows from the inequality $(1-t)m^{x_0}_0 + t m^{x_0}_1 \ge m^{x_0}_t$,
which is implied by \Bthree\ according to Lemma \ref{L:t-convex DASM}.
The last two properties (i) and (ii) enable one to express $u_t$ as a supremum of mountains
$$
u_t (\cdot) =\sup_{x_0 \in \X} m^{x_0}_t (\cdot)+ u_t (x_0),
$$
hence $u_t$ is $b$-convex by the remark immediately preceding \eqref{mountain}.

Conversely, let us show the necessity of \Bthree\ for convexity of $\V^b_{\cl\Y}$.
Using the same notation as above, recall that each mountain $m^{x_0}_i$, $i=0,1$ is $b$-convex.
Assume the linear combination $h_t :=(1-t)m^{x_0}_0 + t\, m^{x_0}_1$ is $b$-convex.
Since $D_x h(x_0) = (1-t) D_x b(x_0, y_0) + t \,D_x b(x_0, y_1)) = D_x m_t (x_0)$, Lemma~\ref{lemma:tangent} requires that $m^{x_0}_t \le h_t $ for every $0\le t \le 1$.
This last condition is equivalent the property characterizing
nonnegative cross-curvature in Lemma \ref{L:t-convex DASM}.
This completes the proof of necessity and the proof of the theorem.
\end{proof}

Let us turn now to the convexity of the principal's problem.

\begin{proof}[Proof of Corollary \ref{C:convex program}]
Corollary \ref{C:convex program} follows by combining the convexity of the set $\U_0$
of feasible strategies proved in Theorem \ref{T:convex set U} with the convexity of $a(q)$ from
Proposition \ref{P:convexity}.  If $\mu$ fails to vanish on each Lipschitz hypersurface,
a little care is needed to deduce convexity of the principal's objective $L(u)$ from that of $a(q)$,
by invoking the conventions adopted in Remark \ref{R:tie-breaker} as follows.
Let $t \in[0,1] \longmapsto u_t = (1-t)u_0 + t u_1$ denote a line segment in the
convex set $\U_0$.  If $q \in \partial u_t(x)$ for some $x \in \X$, then
$y_b(x,q) \in \partial^b u_t(x)$ by Theorem 3.1 of Loeper \cite{Loeper07p};
(a direct proof along the lines of Lemma \ref{L:t-convex DASM} may be found in \cite{KimMcCann07p}).
So $y_b(x,q)$ is among the best responses of $x$ to price menu $v_t=u_t^\bs$.
For each $t\in[0,1]$ select $D u_t(x) \in \partial u_t(x)$ measurably
to ensure $\min \{c(y_b(x,q)) - b(x,y_b(x,q)) \mid q \in \partial u_t(x)\}$
is achieved at $q = Du_t(x)$.  Then
$a(Du_t(x)) \le a((1-t)Du_0(x) + tDu_1(x))$ since
$(1-t)Du_0(x) + t Du_1(x) \in \partial u_t(x)$.
The desired convexity of $L(u)$ follows.
\end{proof}


Next we establish uniqueness of the principal's strategy.

\begin{proof}[Proof of Theorem \ref{T:uniqueness}]
Suppose both $u_0$ and $u_1$ minimize the principal's net losses $L(u)$ on
the convex set $\U_0$. Define
the line segment $u_t = (1-t)u_0 + t u_1$ and --- in case
$\mu$ fails to vanish on each Lipschitz hypersurface ---
the measurable selection $D u_t(x) \in \partial u_t(x)$ as in the proof of
Corollary \ref{C:convex program}.
The strict convexity of $a(q)$ asserted by Proposition \ref{P:convexity}
removes all freedom from this selection.
Under the hypotheses of Theorem \ref{T:uniqueness},
the same strict convexity implies the contradiction
$L(u_{1/2}) < \frac{1}{2}L(u_0) + \frac{1}{2} L(u_1) = L(u_1)$ unless $Du_0 = Du_1$ holds $\mu$-a.e.
This establishes the uniqueness $\mu$-a.e.\ of the agents' equilibrium strategies
$y_{b,v}(x):=y_b(x,Du_1(x))$,
and of the principal's optimal measure $\nu := (y_{b,v})_\# \mu$ in \eqref{push-forward}.

Let $\spt \mu$ denote the smallest closed subset of $\cl X$ containing the full mass of $\mu$.
To identify $u_0=u_1$ on $\spt \mu$ and establish the remaining assertions is more technical.
First observe that
the participation constraint $u_{1/2}(x) \ge b(x,y_\emptyset) - c(y_\emptyset)=:u_\emptyset(x)$
on the continuous function $u_{1/2} \in \U_0$ must bind for some agent type $x_0 \in \spt \mu$;
otherwise for $\e>0$ sufficiently small, $\max\{u_{1/2} -\e,u_\emptyset\}$ would belong to $\U_0$
and reduce the principal's losses by $\e$, contradicting the asserted optimality of $u_{1/2}$.
Since $u_{1/2}$ is a convex
combination of two other functions obeying the same constraint,  we conclude $u_0(x_0)=u_1(x_0)$
coincides with the reservation utility $u_\emptyset(x_0)$ for type $x_0$.
Now use the map $y_{b,v} := y_b \circ Du_1$ from the first paragraph of the proof to
define a joint measure $\gamma := (id \times y_{b,v})_\#\mu$ given by
$\gamma[U \times V] = \mu[U \times y_{b,v}^{-1}(V)]$ for Borel $U \times V \subset \cl \X \times \cl\Y$,
and denote
by $\spt \gamma$ the smallest closed subset $S \subset \cl \X \times \cl\Y$ carrying
the full mass of $\gamma$.
Notice $\spt \gamma$ does not depend on $t \in[0,1]$,  nor in fact
on $u_0$ or $u_1$;  any other optimal strategy for the principal would lead to the same $\gamma$.

Since the graph of $y_{b,v}$ lies in the closed set $\partial^b u_1 \subset \cl \X \times \cl \Y$,
the same is true of $S := \{(x_0,y_\emptyset)\} \cup \spt \gamma$.  Thus $S$ is
$b$-cyclically monotone \eqref{b-cyclical monotonicity} by the result of Rochet \cite{Rochet87}
discussed immediately before Lemma \ref{L:minimal b-convex}.
Lemma \ref{L:minimal b-convex} then
yields a minimal $b$-convex function $u_-$ satisfying $u_-(x_0)=b(x_0,y_\emptyset)-c(y_\emptyset)$
for which $S \subset \partial^b u_-$.
The fact that $(x_0,y_\emptyset) \in S$ implies some mountain
$b(\cdot,y_\emptyset)+\lambda$ bounds $u_-(\cdot)$ from below with contact at $x_0$.
Clearly $\lambda=-c(y_\emptyset)$ whence $u_- \in \U_0$.

Now we have $u_i \ge u_-$ for $i =0,1$ with equality at $x_0$.
Also, $y_{b,v}(x) \in \partial^b u_-(x)$ for $\mu$ almost all $x$, whence $u_-$
must be an optimal strategy: it is smaller in value than $u_i$ and produces at least as favorable
a response as $u_i$ from almost all agents.
Finally since
$$
L(u_i)-L(u_-) \ge \int_X (u_i(x) - u_-(x)) d\mu(x) \ge 0,
$$
the fact that $u_i$ minimizes the losses of the principal implies the continuous integrand
vanishes $\mu$-almost everywhere.  Thus $u_i \ge u_-$ on $\cl \X$,  with equality holding
throughout $\spt \mu$ as desired.

Since $u_0$ was arbitrary,  we have now proved that
all optimal $u\in \U_0$ coincide with $u_1$ on $\spt \mu$.
Optimality of $u$ also implies $\spt \gamma \subset \partial^b u$;
if in addition the participation constraint
$u(x) \ge b(x,y_\emptyset)-c(y_\emptyset)$ binds at $x_0$, then $u \ge u_-$ on $\cl X$.
Although $u_-$ appears to depend on our choice of $x_0 \in \spt \mu$ in the construction above
this is not actually the case: $u(x_0) = u_1(x_0)$ shows the participation constraint binds
at $x_0$ for every optimal strategy and $u_-$ is therefore uniquely determined by its minimality among optimal
strategies $u \in \U_0$.

Now,  since any supremum of $b$-convex functions (not identically infinite)
is again $b$-convex,  define $u_+ \in \U_0$ as the pointwise supremum among all of the
principal's equilibrium strategies $u \in \U_0$.  The foregoing shows $u_+=u_-$ on $\spt \mu$,
while $(x,y) \in \spt \gamma \subset \partial^b u$ implies
\begin{eqnarray*}
u_+(\cdot) \ \ge\  u(\cdot)
&\ge& u(x) + b(\cdot,y) - b(x,y) \\
&=& u_+(x) + b(\cdot,y) - b(x,y)
\end{eqnarray*}
on $\cl\X$, whence $\spt \gamma \subset \partial^b u_+$.  From here we deduce
$L(u_+) \le L(u)$,  hence $u_+$ is itself an optimal strategy for the principal.

Finally, $v:\cl \Y \longmapsto \R \cup \{+\infty\}$ is an equilibrium price menu
in Carlier's reformulation \cite{Carlier01} if and only if $u:= v^b$ minimizes $L(u)$ on $\U_0$,
in which case $u_- \le u \le u_+$ throughout $\cl\X$ implies
$u_+^\bs \le (v^b)^\bs \le u_-^\bs$ throughout $\cl \Y$.
Moreover, $u_-=u_+$ on $\spt \mu$ implies $u_+^\bs = u_-^\bs$ on $\spt \nu$, since
$y_{b,v}(x) \in \partial^b u_\pm(x)$ for $\mu$-a.e.\ $x$ implies
$u_\pm^\bs(y_{b,v}(x)) = b(x,y_{b,v}(x)) - u_\pm(x)$.
We therefore conclude that if $v$ is an equilibrium price menu, then
$v \ge (v^b)^\bs \ge u_+^\bs$ on $\cl \Y$, with both equalities holding $\nu$-a.e.
Conversely, if $v:\cl \Y \longmapsto \R \cup \{+\infty\}$ satisfies $v \ge u_+^\bs$
with equality $\nu$-a.e.,
we deduce the same must be true for its $b$-convex hull
$(v^b)^\bs$,  the latter being the largest $b$-convex function dominated by $v$.
Thus $(v^b)^\bs(y_\emptyset) =c(y_\emptyset)$ and $v^b \in \U_0$ and $v^b \le u^+$ throughout $\cl\X$
with equality holding $\mu$-a.e.  If $\mu$ vanishes on Lipschitz hypersurfaces,
then $D v^b = Du^+$ agree $\mu$-a.e.,  so $L(v^b) = L(u^+)$ and $v^b$ is a optimal
strategy for the principal as desired.
If, on the other hand, $\mu$ does not vanish on all Lipschitz hypersurfaces,
then we may assume $v$ is its own $\bs$-convex hull by Remark \ref{R:tie-breaker}.
Any mountain which touches $u_+^\bs$ from below on $\spt \nu$ also touches
$v \ge u_+^\bs$ from below at the same point,  thus
$\partial^\bs u_+^\bs \subset \partial^\bs v$;
since $v$ is $b$-convex this is equivalent to
$\partial^b u_+ \subset \partial^b v^b$.  This shows the best response of $x$
facing price menu $u_+^\bs$ is also one of his best responses facing price menu $v$:
he cannot have a better response since his indirect utility $v^b \le u^+$.
The constraint on the agent's behaviour imposed by Remark \ref{R:tie-breaker}
now implies $L(v^b) \le L(u^+)$; equality must hold since $u^+$ is one
of the principal's optimal strategies.  This confirms optimality of $v^b$
and concludes the proof of the theorem.
\end{proof}

{
To show stability of the equilibrium 
requires the following convergence result concerning Borel
probability measures $\Prob\big(\cl X \times \cl Y\big)$ on the product space.

\begin{proposition}[Convergence of losses and mixed strategies]\label{P:convergence}
Suppose a sequence of triples $(b_\i,c_\i,\mu_\i) = \lim_{i \to\infty} (b_i,c_i,\mu_i)$
satisfy the hypotheses of Corollary \ref{C:stability}.
Let $L_i(u)$ denote the net losses \eqref{principal's reformulation}
by a principal who adopts strategy $u$ facing data
$(b_i,c_i,\mu_i)$.  If any sequence $u_i$ of $b_i$-convex
functions converge uniformly on $X$,  then their limit $u_\i$ is $b_\i$-convex and
$L_\infty(u_\infty) = \lim_{i \to \infty }L_i(u_i)$.  Furthermore,
there is a unique joint measure $\gamma_\i \in \Prob\big(\cl X \times \cl Y\big)$ supported
in $\partial^{b_\i} u_\i$ with left marginal $\mu_\i$,  and any
sequence of joint measures $\gamma_i \in \Prob\big(\cl X \times \cl Y\big)$
vanishing outside $\partial^{b_i} u_i$ and with left marginal $\mu_i$, must converge weakly-$*$ to
$\gamma_\infty$.
\end{proposition}
}

{
\begin{proof}
Assume a sequence $u_i \to u_\i$ of $b_i$-convex functions converges uniformly on $\cl\X$. 
Topologizing the continuous functions $C\big(\cl Z\big)$ by uniform convergence,
where $Z=\X,\Y$ or $\X \times \Y$,  makes the transformation
$(b,u) \longmapsto u^{b^*}$ given by \eqref{b-transform} continuous on
$C\big(\cl X \times \cl \Y\big) \times C\big(\cl \X\big)$.
This fact allows us
to take $i \to \infty$ in the relation $u_i^{b_i^*b_i}=u_i$
to conclude $b_\i$-convexity of $u_\i$.  From the semiconvexity
\eqref{eq:lip semiconv} of $u_\i$ we infer its domain
of differentiability $\dom Du_\i$ exhausts $\X$ apart from a countable collection
of Lipschitz hypersurfaces,  which are $\mu_\i$-negligible by hypothesis.
Define the map $G_\i(x) = y_{b_\i}(x,Du_\i(x))$ on $\dom Du_\i$.
Since $\partial^{b_\i} u_\i \cap (\dom Du_\i \times \cl \Y)$ coincides
with the graph of $G_\i$,  any measure $\gamma_\infty$ supported in
$\partial^{b_\i} u_\i$ with left marginal $\mu_\infty$ is given \eqref{joint push forward}
by $\gamma_\i := (id \times G_\i)_\#\mu_\i$ as in, e.g.,
Lemma 2.1 of Ahmad et al \cite{AhmadKimMcCann09p}.
This specifies $\gamma_\i$ uniquely.

Now suppose $\gamma_i \ge 0$ 
is a sequence of measures
supported in $\partial^{b_i} u_i$ having left marginal $\mu_i$.
Compactness allows us to extract from any subsequence
of $\gamma_i$ a further subsequence which converges weakly-$*$ to some limit
$\bar \gamma \in \Prob\big(\cl \X \times \cl\Y\big)$.  Since $\mu_i \rightharpoonup \mu_\i$
the left marginal of $\bar \gamma$ is given by $\mu_\i$.  Moreover,
since $u_i(x) + u_i^{b_i^*}(y) \ge b_i(x,y)$ throughout $\cl \X \times \cl \Y$
with equality on $\spt \gamma_i$, uniform convergence of this expression yields
$\spt \bar \gamma \subset \partial^{b_\i} u_\i$.  The uniqueness result of the preceding
paragraph then asserts $\bar \gamma = \gamma_\i$ independently of the choice of subsequence,
so the full sequence $\gamma_i \rightharpoonup \gamma_\infty$ converges weakly-$*$.

Finally, use the measurable selection $Du_i(x) \in \partial u_i(x)$
of Remark \ref{R:tie-breaker} to extend $Du_i(x)$ from $\dom Du_i$ to
$\X$ so as to guarantee that
$G_i(x) := y_{b_i}(x,Du_i(x))) \in \partial^{b_i}u_i(x)$.  Use the Borel map
$G_i:\X \longmapsto \cl\Y$ to push $\mu_i$ forward to the joint probability measure
$\gamma_i := (id \times G_i)_\#\mu_i$ on $\X \times \cl\Y$ defined by
\begin{equation}\label{joint push forward}
\gamma_i[U \times V] := \mu_i[U \cap G_i^{-1}(V)]
\end{equation}
for each Borel $U \times V \subset \X \times \cl\Y$.
Notice $\gamma_i$ is supported in $\partial^{b_i} u_i$
and has $\mu_i$ for its left marginal,  hence converges weakly-$*$ to $\gamma_\infty$.
Moreover, our choice of measurable
selection guarantees that the net losses \eqref{principal's reformulation}
of the principal choosing strategy $u_i$ coincide with
\begin{equation}\label{tlosses}
L_i(u_i) = \int_{X \times \cl \Y} (c_i(y) - u_i^{b_i^*}(y)) d \gamma_i(x,y).
\end{equation}
Weak-$*$ convergence of the measures $\gamma_i \rightharpoonup \gamma_\i$
couples with uniform convergence of the integrands to yield the desired limit
\begin{equation*}
\lim_{i\to \infty}L_{i}(u_{i})
= \int_{X \times \cl \Y} (c_\i(y) - u_\i^{b_\i^*}(y)) d \gamma_\i(x,y) = L_\i(u_\i)
\end{equation*}
and establish the proposition.
\end{proof}
}

{
\begin{proof}[Proof of Corollary~\ref{C:stability}]
Let $\U_{0}^i$ denote the space of $b_i$-convex functions
$u(\cdot) \ge b_i(\cdot,y_\emptyset)-c_i(y_\emptyset)$,
and $L_i(u)$ denote the net loss of the principal who chooses strategy $u$
facing the triple $(b_i, c_i, \mu_i)$.  The $L_i$-minimizing strategies
$u_i \in \U_0^i$ are Lipschitz and semiconvex, with upper bounds \eqref{eq:lip semiconv}
on $|Du_i|$ and $-D^2 u_i$ which are independent of $i$
since $\|b_i - b_\infty\|_{C^2} \to 0$.
The Ascoli-Arzel\`a theorem therefore yields a subsequence $u_{i(j)}$ which converges
uniformly to a 
limit $\bar u$ on the compact set $\cl \X$.
Since the functions $u_i$ have a semiconvexity constant independent of $i$,
it is a well-known corollary that their gradients also converge
$Du_{i(j)}(x) \to D\bar u(x)$ pointwise on the set of common differentiability
$(\dom D\bar u) \cap (\cap_{i=1}^\infty \dom D u_i)$.  This set exhausts
$\cl X$ up to a countable union of Lipschitz hypersurfaces --- which is
$\mu_\i$-negligible by hypothesis.
Setting $G_i(x) = y_{b_{i}}(x,D u_{i}(x))$, it is not hard to deduce
$y_{b_\i}(x,D \bar u(x)) = \lim_{j\to\infty} G_{i(j)}(x)$
on this set from Definition \ref{D:b-exp}.
If we can now prove $\bar u$ minimizes $L_\infty(u)$ on $\U_0^\infty$,
the uniqueness of equilibrium product 
selected  by $\mu_\i$-a.e.\ agent $x\in \X$ in Theorem \ref{T:uniqueness}
will then imply that $\lim_{j\to\infty} G_{i(j)}(x) = G_\infty(x)$ converges
to a limit independent of the subsequence chosen,
hence the full sequence $G_i(x)$
converges $\mu_\i$-a.e.

To see that $\bar u$ minimizes $L_\infty(u)$ on $\U_0^\infty$,
observe $u \in \U_0^\infty$ implies
$u^{b_\infty^*b_i} \in \U_0^i$ is $L_i$-feasible, being the $b_i$-transform of a
price menu $u^{b_\infty^*}(\cdot)$ agreeing with $c_\i(\cdot)$ at $y_\emptyset$.
Moreover, $u^{b_\infty^*b_i} \to u^{b_\infty^*b_\i}$ uniformly as $i \to \infty$
(by continuity of the $b$-transform asserted in the first paragraph of the preceding proof).
The optimality of $u_i$ therefore yields $L_{i}(u_{i}) \le L_i(u^{b_\infty^*b_i})$.
Proposition \ref{P:convergence} allows us to deduce
$L_\infty(\bar u) \le L_\infty(u)$ by taking the subsequential limit $j\to \infty$.
Since the same proposition asserts $b_\i$-convexity of $\bar u$,
we find $\bar u \in \U_0^\i$ is the desired minimizer after taking the limit $j \to \infty$
in $u_{i(j)}(\cdot) \ge b_{i(j)}(\cdot,y_\emptyset)-c_{i(j)}(y_\emptyset)$.
This concludes the proof of $\mu_\i$-a.e.\ convergence of the maps
$G_\infty(x) = \lim_{i\to \infty} G_i(x)$.


Turning to the optimal measures: as in the preceding
proof, a measurable selection $Du_i(x) \in \partial u_i(x)$ consistent with the tie-breaking
hypotheses of Remark \ref{R:tie-breaker}
may be used to extend the Borel map $G_i(x) = y_b(x,D u_i(x))$
from $\dom Du_i$ to $\X$ and define a joint measure
$\gamma_i := (id \times G_i)_\# \mu_i$ supported  on $\partial^{b_i} u_i$
as in \eqref{joint push forward}.
The left marginal of $\gamma_i$ is obviously given by $\mu_i$, and its right marginal
coincides with the unique optimal measure $\nu_i$ given by
Theorem~\ref{T:uniqueness}.  Proposition \ref{P:convergence} then yields weak-$*$
convergence of $\gamma_i \rightharpoonup \gamma_\infty$ and hence of
$\nu_i \rightharpoonup \nu_\infty$.  Theorem~\ref{T:uniqueness} also
asserts the two minimizers $u_\i = \bar u$ agree $\mu_\i$-a.e.  In this case the
uniform limit $\bar u$ is independent of the Ascoli-Arzel\`a subsequence,  hence we recover
convergence of the full sequence $u_i\to u_\i$ in $L^\infty(X,d\mu_\i)$ .
\end{proof}
}

Finally, let us extend Armstrong's desirability of exclusion to our model.
Our proof is inspired by Armstrong's \cite{Armstrong96},
but differs from his in a number of ways.

\begin{proof}[Proof of Theorem \ref{T:Armstrong}]  Use the
$C^3$-smooth diffeomorphism $x \in \cl \X \longmapsto p = D_y
b(x,y_\emptyset) \in \cl X_{y_\emptyset}$ provided by
\Bzero--\Btwo\ and its inverse $p \in \cl X_{y_\emptyset}
\longmapsto x = x_b(y_\emptyset,p) \in \cl \X$ to reparameterize
the space of agents over the strictly convex set $\cl
\X_{y_\emptyset}$. Then $\tu(p) :=
u(x_b(y_\emptyset,p))-b(x_b(y_\emptyset,p),y_\emptyset) + c(y_\emptyset)$ defines
a non-negative $\tb$-convex function, where
$\tb(p,y):=b(x_b(y_\emptyset,p),y) -
b(x_b(y_\emptyset,p),y_\emptyset) + c(y_\emptyset)$. In other words,  the space
$\U_0$ corresponds to the space $\tilde \U_0$ of non-negative
$\tb$-convex functions on $\cl X_{y_\emptyset}$ in the new
parameterization. This subtraction of the reservation utility from
the preference function does not change any agent's response to a
price menu $v$ offered by the principal,  since preferences
between different agent types are never compared. However, it does
make the preference function $\tb(p,y)$ a convex function of $p
\in \cl \X_{y_\emptyset}$, {\dummy as is easily seen}
by interchanging the roles of $x$ and
$y$ in Lemma \ref{L:t-convex DASM}. The indirect utility $\tu(p) =
v^\tb(p)$ is then also convex,  being a supremum
\eqref{b-transform} of such preference functions.

In the new variables,  the distribution of agents $\tf(p)dp = f(x)dx$ is given by
$\tf(p) = f(x_b(y_\emptyset,p)) \det[\partial x_b^i(y_\emptyset,p)/\partial p_j]$.
The principal's net losses $\tL (\tu)=L(u)$ are given as in \eqref{principal's reformulation} by
$$\tL(\tu)= \int_{X_{y_\emptyset}} \ta(D\tu(p),\tu(p),p) \tf(p) dp,$$
where
$\ta(q,s,p)=c(y_\tb(p,q)) - \tb(p,y_\tb(p,q)) + s$ is a convex function of
$q$ on ${\tY_p} := D_p \tb(p,\Y)$ 
for each fixed $p$ and $s$, according to Proposition \ref{P:convexity}; (recall that
$\tb \in C^3 \big(\cl X_{y_\emptyset}\times \cl\Y\big)$
satisfies the same hypotheses \Bzero--\Bthree\ as $b \in C^4 \big(\cl\X \times \cl\Y\big)$,
except for the possibitity that four continuous derivatives with respect to variables
in $X_{y_\emptyset}$ fail to exist, which is irrelevant as already discussed).  This
convexity implies
$$
\ta(q,s,p) \ge \ta(q_0,s,p) + \langle D_q \ta(q_0,s,p),q-q_0 \rangle
$$
for all $q,q_0 \in \cl \tY_p$.
With $p$ still fixed, the choice
$q_0 = D_p \tb(p,y_\emptyset) = \zero$
shows $\ta(\zero,s,p) =s$ whence
$\ta(q,s,p) \ge \langle D_q \ta(\zero,s,p), q \rangle$ for $s = \tu(x) \ge 0$.

Now suppose $\tu \in \tilde \U_0$ minimizes $\tL(\tu)$.
For $\e \ge 0$, define the continuously increasing
family of compact convex sets
$\tU_\e := \{p \in \cl X_{y_\emptyset} \mid \tu(p) \le \e\}$.
Observe that $\tU_0$ must be non-empty, since
otherwise for $\e>0$ small enough $\tU_\e$ would be empty,
and then $\tu-\e \in \tilde \U_0$ is a better strategy,
reducing the principal's losses by $\e$.
We now claim the interior of the set $\tU_0$ --- which corresponds to agents
who decline to participate --- contains a non-zero fraction of the total
population of agents.  Our argument is inspired by
the strategy Armstrong worked out in a special case \cite{Armstrong96},
which was to show that unless this conclusion is true,
the profit the principal extracts from agents in $\tU_\e$
would vanish at a higher order than $\e>0$,
making $\tu_\e := \max\{\tu - \e,0\} \in \tilde \U_0$ a better strategy
than $\tu$ for the principal when $\e$ is sufficiently small.

For $\e>0$, the contribution of $\tU_\e$ to the principal's profit is given by
\begin{eqnarray}\label{div thm profit bound}
-\tL_\e(\tu) &:=&  -\int_{\tU_\e} \ta(D\tu(p),\tu(p),p) \tf(p) dp
\cr &\le& -\int_{\tU_\e} \langle D_q \ta(\zero,\tu(p),p), D\tu(p) \rangle \tf(p) dp
\cr &=& \int_{\tU_\e} \tu(p) \nabla_p \cdot (\tf(p) D_q \ta(\zero,\tu(p),p) ) dp
  - \int_{\partial \tU_\e} \tu(p) \langle D_q \ta , \hat n \rangle \tf(p) dS(p)
\end{eqnarray}
where $\hat n = \hat n_{\tU_\e}(p)$
denotes the outer until normal to $\tU_\e$ at $p$,
and the divergence theorem has been used.
Here $\partial \tU_\e$ denotes the boundary of the convex set $\tU_\e$,
and $dS(p)$ denotes the $n-1$ dimensional surface (i.e.\ Hausdorff)
measure on this boundary.  {\dummy (For Sobolev functions, the integration by
parts formula that we need is contained in \S 4.3 of \cite{EvansGariepy92}
under the additional restriction that the vector field
$\tu(\cdot) D_q a(\zero, \tu(\cdot),\cdot)$ be $C^1$ smooth, but
extends immediately to Lipschitz vectors fields by approximation;  the
operation of restricting $\tf$ to the boundary of $\tU_\e$ is there shown to give a bounded
linear map from $W^{1,1}(U_\e,dp)$ to $L^1(\partial U_\e,dS)$ called the boundary {\em trace}.)}
As $\e \to 0$, we claim both integrals in \eqref{div thm profit bound}
vanish at rate $o(\e)$ if the interior of $\tU_0$ is empty.
To see this,  note $\tu=\e$ on $\partial \tU_\e \cap \intr X_{y_\emptyset}$,
so
\begin{eqnarray*}
&&\int_{\partial \tU_\e} \tu(p) \langle D_q \ta , \hat n \rangle
\tf(p) dS(p) \\
&=&\e \int_{\partial \tU_\e} \langle D_q \ta , \hat n
\rangle \tf(p) dS(p)
+ \int_{\partial \tU_\e \cap \partial X_{y_\emptyset}}
[\tu(p)-\e] \langle D_q \ta , \hat n \rangle
\tf(p) dS(p)\\
&=&\e \int_{\tU_\e} \nabla_p \cdot (\tf(p) D_q \ta(\zero,\tu(p),p)
) dp
+ \int_{{\dummy \tU_\e} \cap \partial X_{y_\emptyset}}
[\tu(p)-\e] \langle D_q \ta , \hat n \rangle \tf(p) dS(p).
\end{eqnarray*}
Since $0 \le \tu \leq \e$ in $\tU_\e$, we combine the last inequality with
\eqref{div thm profit bound} to obtain
\begin{equation}\label{ratio}
-\frac{\tL_\e(\tu)}{\e}
\leq 
\int_{\tU_\e}
\bigl|\nabla_p \cdot (\tf(p) D_q \ta(\zero,\tu(p),p) )\bigr| dp
+ \int_{\tU_\e \cap \partial X_{y_\emptyset}}
\bigl|\langle D_q \ta , \hat n \rangle \tf(p) \bigr| dS(p).
\end{equation}
{\dummy Notice that domain monotonicity implies the $\e \to 0$ limit of the last expressions above
is given by integrals over the limiting domain $\tU_0 = \cap_{\e>0} \tU_\e$.
Assume now the interior of the convex set $\tU_0$ is empty,
so that $\tU_0$ has dimension at most $n-1$. Then the volume} $|\tU_\e| =
o(1)$, hence the first integral in the right hand side
dwindles to zero as $\e\to 0$,
(recalling that  $\tu$ is Lipschitz, $\tf \in W^{1,1}$ and $\ta \in C^3$).
Concerning the second term,  {\dummy if the convex set $\tU_0$ has dimension $n-1$
then its {\em relative} interior}
must be disjoint from the boundary of the convex body $X_{y_\emptyset}$, since the
latter is assumed to have no $n-1$ dimensional facets.
{\dummy Either way} $\tU_0 \cap
\partial X_{y_\emptyset}$ has dimension at most $n-2$, which
implies that
$$
\int_{\tU_\e \cap \partial X_{y_\emptyset}} dS(p)=o(1)
$$
as $\e \to 0$.
All in all, we have shown $L_\e(\tu) = o(\e)$ as $\e \to 0$ whenever $\tU_0$ has empty interior,
which --- as was explained above --- contradicts the asserted optimality of the strategy $\tu$.
However, even if $\tU_0$ has non-empty interior,
more must be true to avoid inferring the contradictory conclusion $L_\e(\tu) = o(\e)$ as
$\e \to 0$ from \eqref{ratio}: 
one of the two limiting integrals
$$ 
\int_{\tU_0} |\nabla_p \cdot (\tf(p) D_q \ta(\zero,\tu(p),p) )|
dp>0 \quad {\rm or} \quad \int_{\tU_0\cap \partial
X_{y_\emptyset}} |\langle D_q \ta , \hat n \rangle| \tf(p) dS(p)
>0
$$ 
must be non-vanishing.
In either case, the $W^{1,1}$ density $\tf$ must be positive somewhere in $\tU_0$,
whose interior therefore includes a positive fraction of the agents.  Since $\tu$ is differentiable
with vanishing gradient on the interior of $\tU_0$,
there is no ambiguity in the strategy of these agents: they respond to $\tu$ by
choosing the null product.
\end{proof}

\appendix
\section{Minimal $b$-convex potentials}

The purpose of this appendix is to establish a mathematical result (and
some terminology) needed in the last part of the uniqueness proof, Theorem \ref{T:uniqueness}.
In particular, we establish a minimality property
enjoyed by Rochet's construction of a $b$-convex function for which
$\partial^b u$ contains a prescribed set \cite{Rochet87};  Rochet's construction is modeled
on the analogous construction by Rockafellar of a convex function $u$ whose
subdifferential $\partial u$ contains a given cyclically monotone set \cite{Rockafellar66}.

Recall a relation $S \subset \cl\X \times \cl\Y$ is
{\em $b$-cyclically monotone} if for each integer $k \in \N$
and $k$-tuple of points $(x_1,y_1), \ldots, (x_k, y_k) \in S$,
the inequality
\begin{equation}\label{b-cyclical monotonicity}
\sum_{i=1}^k b(x_i,y_i) - b(x_{i+1},y_{i}) \ge 0
\end{equation}
holds with $x_{k+1}:=x_1$.  For a function $u:\cl \X \longmapsto \R \cup \{+\infty\}$,
the relation $\partial^b u \subset \cl\X \times \cl\Y$ consists of those points $(x,y)$ such that
\begin{equation}\label{b-subdifferential}
u(\cdot) \ge u(x) + b(\cdot,y) - b(x,y)
\end{equation}
holds throughout $\cl \X$.  Rochet's generalization
of Rockafellar's theorem asserts that $S \subset \cl \X \times \cl \Y$ is
$b$-cyclically monotone if and only if there exists a $b$-convex function
$u:\cl X \longmapsto \R \cup \{+\infty\}$
such that $S \subset \partial^b u$; see also
\cite{GangboMcCann96} \cite{Levin99} \cite{Ruschendorf96}.
Here we need to extract a certain minimality property from its proof.

\begin{lemma}\label{L:minimal b-convex}
Given a $b$-cyclically monotone $S \subset \cl \X \times \cl \Y$  and
$(x_0,y_0) \in S$,  there is a $b$-convex function $u$ vanishing at $x_0$
and satisfying $S \subset \partial^b u$,  which is minimal in the sense that
$u \le \tu$ for all $\tu:\cl \X \longmapsto \R\cup\{+\infty\}$ vanishing at $x_0$
with $S \subset \partial^b \tu$.
\end{lemma}

\begin{proof}
Given a $b$-cyclically monotone $S \subset \cl \X \times \cl \Y$  and
$(x_0,y_0) \in S$,  Rochet \cite{Rochet87} verified the elementary fact
that the following formula defines a $b$-convex function $u$ for which
$S \subset \partial^b u$:
\begin{equation}\label{Rochet-Rockafellar}
u(\cdot) = \sup_{k \in \N} \sup_{(x_1,y_1),\ldots,(x_k,y_k) \in S} b(\cdot, y_k) - b(x_0,y_0) + \sum_{i=1}^k b(x_i,y_{i-1}) - b(x_i,y_{i}).
\end{equation}
Taking $k=0$ shows $u(x_0) \ge 0$,  while the opposite inequality $u(x_0)\le 0$ follows
from $b$-cyclical monotonicity \eqref{b-cyclical monotonicity} of $S$.
Now suppose $\tu(x_0) = 0$ and $S \subset \partial^b \tu$.
For each $k\in \N$ and $k$-tuple in $S$,  we claim $\tu(\cdot)$ exceeds the expression under
the supremum in \eqref{Rochet-Rockafellar}.  Indeed,  $(x_i,y_i) \in S \subset \partial^b \tu$
implies
$$\tilde u(x_{i+1}) \ge \tilde u(x_i) + b(x_{i+1},y_i) - b(x_i,y_i).
$$
and $\tu(x_i)<\infty$,  by evaluating \eqref{b-subdifferential} at $x_{i}$ and at $x_0$.
Summing the displayed inequalities from $i=0,\ldots,k$, arbitrariness of $x_{k+1} \in \cl \X$
yields the desired result: $\tu(x_{k+1}) \ge u(x_{k+1})$.
\end{proof}




\end{document}